%% file: main.tex
\documentclass[12pt]{article}

\input{preamble}
\input{macros}

\begin{document}
\input{title}

\begin{abstract}
    In this paper, we propose a new model of chemotaxis motivated by ant trail pattern formation, formulated as a coupled parabolic-parabolic local PDE system, for the population density and the chemical field. The main novelty lies in the transport term of the population density, which depends on the second-order derivatives of the chemical field. This term is derived as an anticipation-reaction steering mechanism of an infinitesimally small ant as its size approaches zero. We establish global-in-time existence and uniqueness for the model, and the propagation of regularity from the initial data. Then, we build a numerical scheme and present various examples that provide hints of trail formation.
\end{abstract}

\tableofcontents

\section{Introduction}
\input{Introduction}

\section{Derivation of the model}
\label{sec:DerMod}
\input{Derivation}

\section{Main results}
\label{sec:MainRes}
\input{AnalysisChemotaxisParabolic}

\section{Numerical simulations}
\label{sec:NumSim}
\input{NumericSim}

\input{Appendix}

\bibliographystyle{unsrt}  
\bibliography{Bibliography}

\end{document}

%% file: preamble.tex
\usepackage[utf8]{inputenc}

\usepackage[english]{babel}
\usepackage[top=2cm,bottom=2cm,left=2cm,right=2cm,marginparwidth=1.75cm]{geometry}
\usepackage{amsmath}
\usepackage{amsthm}
\usepackage{amsfonts}
\usepackage{float}
\usepackage{amssymb}
\usepackage{mathtools}
\usepackage{upgreek}
\usepackage{mathtools}
\usepackage{bmpsize}
\usepackage{bbm}
\usepackage{csquotes}
\usepackage{indentfirst}
\usepackage{scalerel}
\usepackage[toc,title,page]{appendix}
\usepackage{lipsum}
\usepackage[colorlinks=true, allcolors=blue]{hyperref}
\usepackage{caption}

\newtheorem{theorem}{Theorem}[section]
\newtheorem{proposition}[theorem]{Proposition}

\newtheorem{corollary}[theorem]{Corollary}
\newtheorem{lemma}[theorem]{Lemma}
\theoremstyle{definition}
\newtheorem{definition}[theorem]{Definition}

\newenvironment{italicsremark}{\begin{remark}\itshape}{\end{remark}}

\newtheorem{assumption}{Assumption}
\newtheorem{assumptionalt}{Assumption}[assumption]
\newenvironment{assumptionp}[1]{
  \renewcommand\theassumptionalt{#1}
  \assumptionalt
}{\endassumptionalt}

\theoremstyle{remark}
\newtheorem*{remark}{Remark}

\theoremstyle{definition}

\newcommand*\samethanks[1][\value{footnote}]{\footnotemark[#1]}

\date{}


%% file: macros.tex
\DeclareMathOperator*{\esssup}{ess\,sup}
\newcommand{\defeq}{\mathrel{\mathop:}=}

\newcommand{\mapto}{ \xrightarrow[]{}}
\renewcommand{\div}{\text{div}}

\newcommand{\dR}{\mathbb{R}}

%% file: title.tex
\title{Curvature in chemotaxis:\\ A model for ant trail pattern formation}
\author{Charles Bertucci\thanks{CMAP, CNRS, École polytechnique, Institut Polytechnique de Paris, 91120 Palaiseau, France}, Matthias Rakotomalala\samethanks, Milica Toma\v{s}evi\'{c}\samethanks}

\maketitle

%% file: Introduction.tex
\textit{Chemotaxis} is the process of movement in response to chemical stimuli. This phenomenon can be observed at the scale of bacteria, cells, or insects. When the chemical is produced by the population sensing it, collective behavior emerges. In the celebrated paper of Keller and Segel \cite{keller1971model}, the authors proposed a PDE system modeling both the evolution of a chemical and the density at the macroscopic scale of cells, that are attracted by the concentration of the chemical they produce. The minimal version of the proposed model has the following form:
\begin{align*}
        \partial_t \rho &= \Delta_x \rho - \chi \nabla_x \cdot (\nabla_x c \rho), \text{ in }(0,T)\times \dR^2\\
        \partial_t c &= \Delta_x c - \gamma c + \rho,\text{ in } (0,T)\times \dR^2
\end{align*}
where $c$ is the concentration of the chemo attractant, $\rho$ the density of cells and $\chi$ the force of interaction, with initial data $\rho_0,c_0$. The chemoattractant diffuses and evaporates, and its production is proportional to the density of cells. The cells are diffusing in space and the transport term $\nabla_x c$ models the attraction of the cells towards the points of high concentrations. 
This parabolic-parabolic version of the Keller-Segel model has now been extensively studied in the literature, see \textit{e.g} \cite{calvez2008parabolic,biler20068pi, CORRIAS2014, Mizo20}. Contrary to the elliptic version of the model, in this case, there is no sharp threshold for global existence \textit{vs} finite time blow-up for a reasonable class of initial data. However, we do expect, when $\chi$ is large and the support of $\rho_0$ small, a finite-time blowup. Recently it has been shown that this model can be derived as the mean-field limit of interacting particles under a smallness condition on $\chi$ in $\dR^2$ and it has now been rigorously proven that the particle system converges towards a Mckean-Vlasov process \cite{ fournier2023particle}.
For a review of chemotaxis models, we refer to \cite{perthame2004pde}.

The \textit{Formicidae} family, which includes all ant species, is known for its complex collective behaviors. They can be observed forming trails connecting food sources to the nest. These trails are formed using chemical markers (\textit{pheromones}), emitted by individual ants and detected through their antennae. This chemotaxis phenomenon has peculiar qualitative properties, differing from the Keller-Segel framework, allowing the colony to coordinate and form complex trail structures. This process is known as \textit{stigmergy} \cite{grasse1982termitologia}, a concept that has been extensively studied in biological literature, see \textit{e.g.,} \cite{ATTYGALLE19851, holldobler1995chemistry, poissonnier2019experimental,beckers1990collective, detrain2001influence, edelstein1994simple}. The complexity of ant stigmergy raises new modeling challenges, and there is ongoing literature on the subject. 

A discrete approach consists of modeling the movement of particles as a biased random walk on a discrete lattice
\cite{deneubourg1990self,
vincent2004effect}. In \cite{boissard2013trail}, the authors study a model where both the ants and the pheromones they deposit are treated as discrete particles, with the ants adjusting their movement orientation based on the surrounding pheromone particles.

For continuous models, Amorim \cite{amorim2015modeling} proposed a PDE system that shows hints of trail formation in the presence of attractive food sources. Other studies, such as \cite{calenbuhr1992modelI, calenbuhr1992modelII, couzin2003self}, focus on the role of antennae in the sensing-reaction mechanism of ants to chemical stimuli. These studies emphasize the types of scalar fields, derived from the chemical concentration, that are necessary for the steering mechanism of a particle to follow a trail of chemoattractant. Fontelos and Friedman \cite{fontelos2015pde} proposed a PDE model and proved the existence of trails if the interaction intensity with the field is sufficiently large. In \cite{bruna2024lane}, the authors derived a model where each particle senses the gradient of the concentration field ahead of its position. This anticipation term, which could be understood as the length of an antenna, is essential for lane formation.


In this paper, we propose a new PDE model for ant trail pattern formation. The ants are seen as particles that are modeled through a position $x$ in space and an angle $\theta$ of orientation. They move at a constant speed in their direction $\theta$ and diffuse in space. The main novelty resides in the dependence of the drift of the angle in the second-order derivative (\textit{curvature}) of the chemotactic field. 

To illustrate the necessity of this term, imagine the chemo-attractant field $c$ as a mountain range, where the altitude represents the concentration of the chemo-attractant. A trail is analogous to a mountain ridge, a curve along which there is a high concentration—a local maximum. At such a crest, the gradient is zero. To distinguish between a crest, a summit, and a valley, we need second-order information, specifically the bi-dimensional curvature of the terrain at that point, which the hessian of the chemo-attractant provides: $\nabla^2_x c$. The Hessian allows us to understand the local geometry of the field, enabling the ant to follow trails accurately.

\begin{figure}
    \centering
    \includegraphics[width=.9\linewidth]{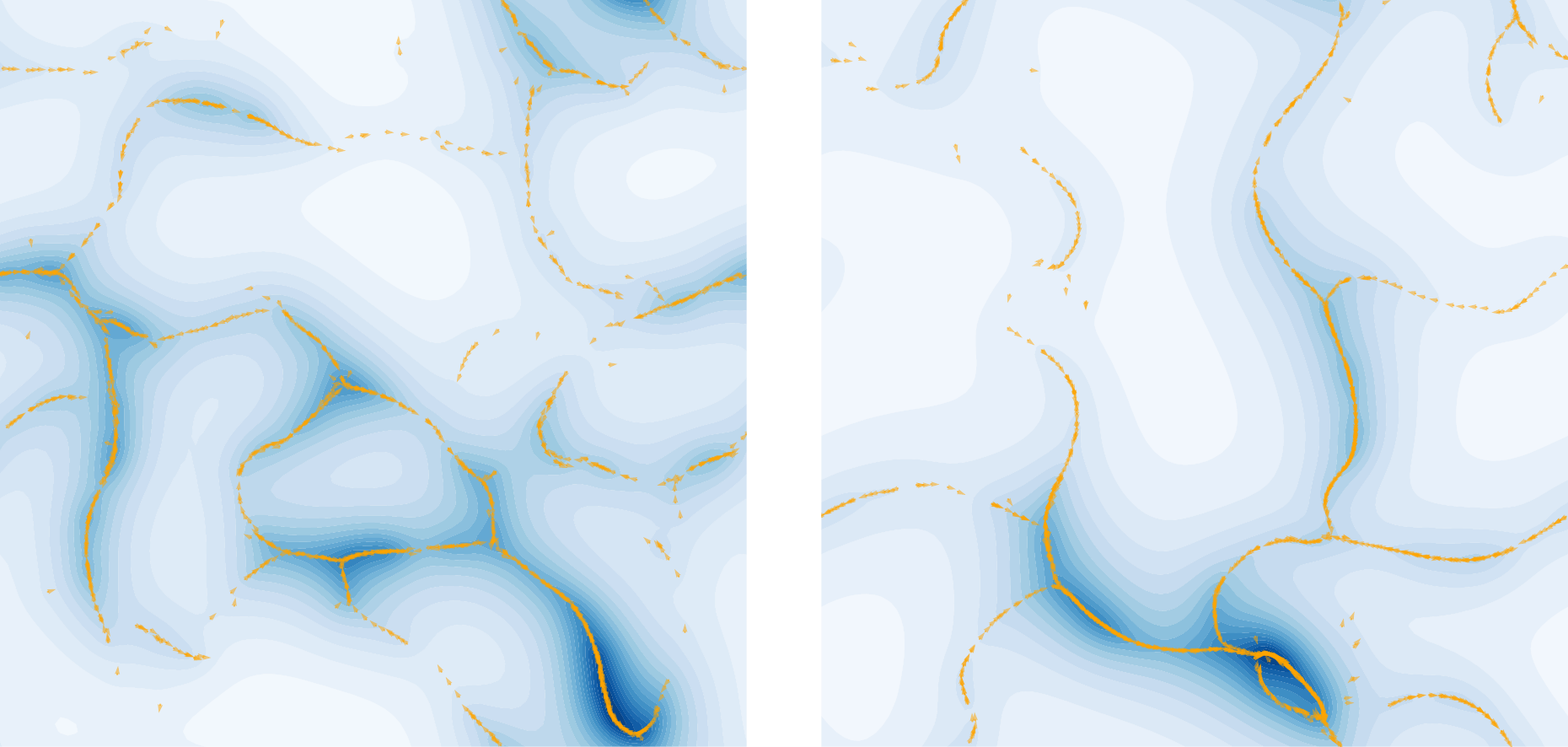}
    \caption{Typical Monte-Carlo particle simulations of our McKean-Vlasov model. The \textit{orange} dots represent the particles and in \textit{blue} the concentration field of pheromones.}
    \label{fig:IntroductionSimulationResult}
\end{figure}

To incorporate the Hessian, we suppose that the density of particles diffuses in $\theta$ and that the following scalar field transports the orientation:
\begin{equation*}
    B(\theta,\nabla_x c(x), \nabla^2_x c(x)) = v^\perp(\theta)\nabla_x c(x) + \tau v^\perp(\theta)\cdot \nabla^2_x c(x) v(\theta),
\end{equation*}
where $\tau \geq 0$ represents the level of anticipation of the ant, and where $v(\theta)$(resp. $v^\perp(\theta)$) are defined as $(\cos(\theta), \sin(\theta)) \in \dR^2$ (resp. $(-\sin(\theta),\cos(\theta))$). Given this scalar field, we couple the system in a manner similar to the Keller–Segel model, incorporating a chemical field $c$, and obtain:
\begin{align*}
    \partial_t \rho &= - \lambda v\cdot \nabla_x\rho - \chi \partial_\theta(B(\nabla_x c, \nabla^2_xc)\rho) + \sigma_\theta \Delta_\theta \rho + \sigma_x \Delta_x \rho \text{ on } (0,T)\times \mathbb{\dR}^2\times \mathbb{T}_{2\pi},\\
    \partial_t c &= -\gamma c + \sigma_c \Delta_x c +\mu \int \rho d\theta \text{ on }(0,T)\times \mathbb{\dR}^2,\\
    \rho_{t=0} &= \rho_0, c_{t=0} = c_0,
\end{align*}

where $\mathbb{T}_{2\pi}$ is the $2\pi$-periodic Torus. This system is the main focus of this paper. In Section~\ref{sec:DerMod}, it is derived as the limit of an anticipation-reaction mechanism performed by an ant using its two sensory antennae. The differential quantity is obtained by taking the limit as the size of the antennae approaches zero. The dependence on the second-order derivatives of $c$ allows us to derive a \textit{local} PDE model while still achieving trail formation (see Figure~\ref{fig:IntroductionSimulationResult} for a typical numerical simulation of our model). Setting $\tau =0$ would result in aggregation towards a center of mass, similarly as in the Keller–Segel model (\textit{cf.} \cite{bruna2024lane}), rather than trail formation. 

Our model is a singular, attractive coupled PDE system. However, under a uniform ellipticity condition, we can obtain uniform bounds on the average in $\theta$ of the density due to the boundedness of the spatial speed (see Lemma~\ref{lem:average}). This estimate prevents finite-time blow-up, allowing for a global-in-time existence and uniqueness result (see Theorem~\ref{thm:existuniqPara}). One could wonder about the behavior of the model for $\sigma_x = 0$, since the operator $-v \cdot \nabla_x + \Delta_\theta$ satisfies H\"ormander's condition, implying hypoelliptic regularity. Nevertheless, deriving estimates on the singularly coupled system remains a nontrivial challenge, and we leave it for future works.

The PDE model is naturally associated with a system of interacting particles representing a colony of $N$-particles (see Section~\ref{sec:DerMod} and Section~\ref{sec:NumSim} for a detailed presentation). The particle model is written as the following stochastic differential equation system:
\begin{align*}
        dX^{i}_t &= \lambda v(\Theta^i_t)dt + \sqrt{2\sigma_x}dW^{1,i}_t,\\
        d\Theta^i_t &= \chi B(\Theta^i_t, \nabla c^N(t, X^i_t), \nabla^2 c^N(t, X^i_t))dt + \sqrt{2\sigma_\theta}dW^{2,i}_t,\\
        \partial_t c^N &= -\gamma c^N + \sigma_c \Delta c^N + \mu m^N,\\
        c_{t=0} & = c_0,(X^i_0,\Theta^i_0) \ i.i.d. \sim \rho_0,
\end{align*}

for $1\leq i\leq N$, where $(X_t^i,\Theta^i_t)$ represents the state of the particle $i$ at time $t$, $(W^{1,1,i}, W^{1,2,i}, W^{2,i})$ is a $3$-dimensional Brownian motion, and $m^N$ is the empirical measure of the spatial position,

\begin{equation*}
    m^N \defeq \frac{1}{N}\sum^N_{i=1} \delta_{X^i_t}.
\end{equation*}

This raises the problem of convergence of the particle system (see Section~\ref{sec:DerMod}) towards the well-posed McKean-Vlasov equation (Theorem~\ref{thm:MckVExistuniq}) and the propagation of chaos, that we plan on addressing in future research.

\paragraph{Plan of the paper} In Section~\ref{sec:DerMod}, we derive our model, the McKean-Vlasov equation and the particle system. In Section~\ref{sec:MainRes}, we announce the results proved in this paper, namely, global in-time existence and uniqueness result together with the propagation of regularity of the initial data. In Section~\ref{sec:LinFP}, we show some preliminary results on the linear Fokker-Planck equation. And in Section~\ref{sec:ProofMR}, we prove the theorems stated in Section~\ref{sec:MainRes}. Finally, in Section~\ref{sec:NumSim}, we provide numerical results for the model at both the particle and macroscopic scales, which show hints of trail formation—where particles agglomerate along curves of high concentration of the chemical field and move tangentially along these paths.

%% file: Derivation.tex
The derivation of the model is split into two parts. First, we focus on the microscopic derivation observing the natural emergence of second derivatives of the pheromone concentration field. Then, we formally take the mean-field limit of the microscopic model and obtain the PDE system, the main focus of this work. In addition, we present an extension of the model with a two-state population. 

\subsection{Our microscopic model} 
We first derive the dynamic of an infinitesimal small ant at the micro-scale given a concentration $c : \dR^2 \mapto \dR_+ $ of chemo-attractant. Secondly, having a finite number of ants, we model the dynamic of the chemical field they produce which introduces interactions between ants. This leads us to a microscopic system of interacting particles.

\paragraph{Movement of one ant given a pheromone field}
We model an ant through a position $X\in\dR^2$ in the plane and an orientation $\Theta \in [0,2\pi)$. Let $((\Omega,(\mathcal{F}_s)_s,\mathbb{P}),(W^{1,1},W^{1,2},W^2))$ be a filtered probability space with the usual hypothesis, equipped with a 3-dimensional Brownian motion $(W^{1,1},W^{1,2},W^2)$. We denote by $W^1 = (W^{1,1},W^{1,2})$, the first two components, forming a 2-dimensional Brownian motion. In the whole paper, we will use the following notation: for $\theta \in [0,2\pi)$, 

\begin{equation}
        \label{def:vspeed}
        v(\theta) \defeq \begin{pmatrix}
                        \cos(\theta)\\
                        \sin(\theta)\\
                    \end{pmatrix}, 
        \frac{d}{d\theta}v(\theta) = v^\perp(\theta) \defeq   \begin{pmatrix}
                        -\sin(\theta)\\
                        \cos(\theta)\\
                    \end{pmatrix} \in \dR^2.
\end{equation}

The ant moves in the plane according to the following stochastic differential equation,
\begin{equation*}
    dX_t = \lambda v(\Theta_t)dt + \sqrt{2\sigma_x}dW^1_t,
\end{equation*}

since $\Theta_t$ is the azimuthal angle, $v(\Theta_t)$ represents the unit direction in which the particle moves. The parameter $\lambda>0$ represents the constant speed of an ant, and $\sigma_x \geq 0$ is the diffusion coefficient interpreted as a perturbation of the motion by external random forces.

In order to derive the main modeling novelty of this paper, let us consider the following biological facts. The \textit{Formicidae} family possesses two antennae at the front of their heads, serving sensory purposes, including the detection of chemo-attractants left by their colony. These chemical markers indicate trails towards areas of interest such as food sources or the nest. The extremal part of each antenna consists of multiple flagellar segments, forming sensor groups along its length. Let $\varepsilon$ denote the infinitesimal small length of an antenna, and let $\pi/2>\beta>0$ represent half of the angle between the two antennae. Thus, the midpoint of the left antenna is at $X+\varepsilon v(\Theta + \beta)$, and that of the right antenna is at $X+\varepsilon v(\Theta - \beta)$. In Figure \ref{fig:AntAnticipation}, we summarise the modeling quantities in a schematic representation of an ant.

We suppose that each of the antennae is able, using its multiple flagellars, to sense the gradient of the concentration field $c$ at its respective position. We denote by $A^\varepsilon_{right}$ (resp. $A^\varepsilon_{left}$) the stimulus perceived by the $right$ (resp. $left$) antenna. The stimuli are expressed as the following Taylor expansions of the concentration field:
\begin{align*}
    A^\varepsilon_{left} &= c(X_s+\varepsilon v(\Theta_s+\beta)) + \tau v(\Theta_s) \cdot \nabla_x c(X_s + \varepsilon v(\Theta_s+\beta)),\\
    A^\varepsilon_{right} & = c(X_s+\varepsilon v(\Theta_s-\beta)) + \tau v(\Theta_s) \cdot \nabla_x c(X_s + \varepsilon v(\Theta_s-\beta)),
\end{align*}

for some $\tau > 0$. These signals are interpreted as anticipations of the chemical field. They approximate the concentration of chemo-attractant that the ants would encounter after a short period $\tau/\lambda$ if they continue moving in their current direction, starting from the corresponding sensor. We call $\tau$ the \textit{anticipation rate}.

The ants try to stay along curves of high concentration, they react to the sensory inputs by correcting their orientation $\Theta$. If the \textit{left}-stimulus is stronger than the right one: $A^\varepsilon_{left} > A^\varepsilon_{right}$, then the variation of $\Theta$ should be positive to move in the trigonometric direction.
Similarly, if $A^\varepsilon_{left} < A^\varepsilon_{right}$, then the variation of the angle should be negative. Rescaling by $\varepsilon$, we obtain after a small time $\delta t$ the following variation $\delta \Theta$ of angle:

\begin{equation*}
     \delta \Theta = \chi \frac{A^\varepsilon_{left} -A^\varepsilon_{right} }{\varepsilon}\delta t + \sqrt{\delta t} \eta,
\end{equation*}

where $\chi>0$ is the force of reaction, and $\eta$ is a Gaussian noise. Performing a Taylor expansions w.r.t $\varepsilon$ of the terms $A^\varepsilon_{left}, A^\varepsilon_{right}$ and taking the limit as $\varepsilon$  goes to zero, we obtain the following limiting anticipation-reaction drift:

\begin{equation*}
    \label{eq:DerivationDrift}
    \lim_{\varepsilon \mapto 0} \frac{A^\varepsilon_{left} -A^\varepsilon_{right} }{\varepsilon} = (v(\Theta + \beta)+v(\Theta - \beta))\cdot \big[ \nabla_x c(X_s) + \tau \nabla^2_x c(X_s) v(\Theta_s)\big].
\end{equation*}

\begin{figure}
    \centering
    \includegraphics[width=.6\linewidth]{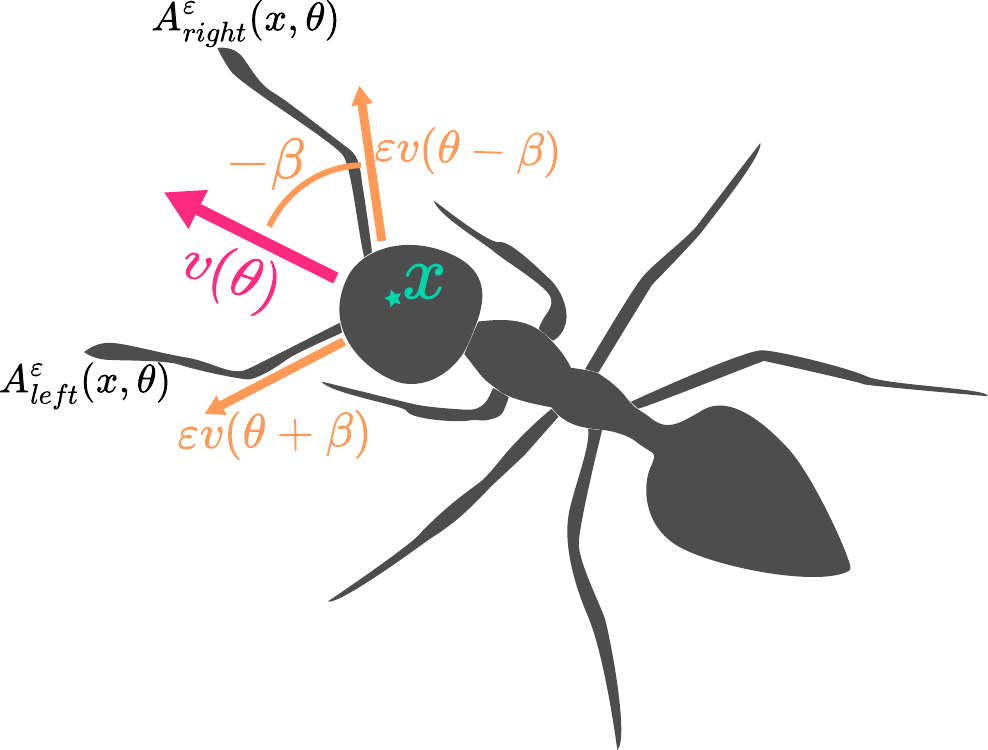}
    \caption{Diagram representing a schematic ant with the different quantities involved in the process of anticipation-reaction at scale $\varepsilon$.}
    \label{fig:AntAnticipation}
\end{figure}

From here, using trigonometric identities, one can check that:

\begin{equation*}
    v(\Theta+\beta) + v(\Theta-\beta) = \sin(\beta) v^\perp(\Theta).
\end{equation*}

So that we then obtain the following drift:
\begin{equation*}
    \sin(\beta)(v^\perp(\Theta)\nabla_x c(X_s) + \tau v^\perp(\Theta)\cdot \nabla^2_x c(X_s) v(\Theta_s).
\end{equation*}

We thus define the drift function $B:(0,2\pi]\times \dR^2\times \mathcal{M}_{2,2}(\dR) ; \mapto \dR$ as:
\begin{equation}
    \label{eq:ThetaDrift}
    \tag{$B_\theta$}
    B(\theta,p, A) = v^\perp(\theta) \cdot p + \tau v^\perp(\theta) \cdot A v(\theta),
\end{equation}

for $\theta\in (0,2\pi]$, $p\in \dR^2$ and $A \in M_{2,2}(\dR)$. Absorbing $\sin(\beta)$ in $\chi$, the dynamic for $\Theta$ then writes as:
\begin{equation*}
    d\Theta_t = \chi B(\Theta_t, \nabla_x c(X_t), \nabla^2_x c(X_t)) dt + \sqrt{2\sigma_\theta}dW^2_t,
\end{equation*}
where $\sigma_\theta > 0$ stands for the azimuthal diffusion coefficient. And the equation lies in the tangent space of the $2\pi$-torus where $\Theta$ takes its values. The parameter $\sigma_\theta$ models the exploration behavior. In the absence of chemical stimuli, an ant will engage in random exploration of the domain.

Summing up, given a concentration field ${c:[0,T]\times \dR^2\mapto \dR}$, that we now allow to depend on time, the dynamic of an ant is given by
\begin{equation}
    \label{sde:Antmicro}
    \begin{cases}
        dX_t = \lambda v(\Theta_t) dt + \sqrt{2\sigma_x}dW^1_t,\\
        d\Theta_t = \chi B(\Theta_t, \nabla c(t, X_t), \nabla^2c(t, X_t)) dt + \sqrt{2\sigma_\theta}dW^2_t,
    \end{cases}
\end{equation}

with $B$ defined as in \eqref{eq:ThetaDrift}. This stochastic differential equation(SDE) is non-potential, which challenges the analysis of its asymptotic behavior. Nevertheless, we can interpret it as the drift associated with maximizing locally the second-order Taylor approximation of the chemotactic field. Indeed, let us make the following remarks on the dynamics of the angle.

\paragraph{Autonomous azimuthal equation} For fixed $p\in \dR^2$ and $A \in M_{2,2}(\dR)$, the law associated with the autonomous dynamic:
\begin{equation}
    \label{sde:PhiAutomomous}
    d\Phi_t = \chi B(\Phi_t, p, A)dt + \sqrt{2}dW^2_t,
\end{equation}
is potential, and attracted by the following explicit stationary distribution:
\begin{align}
    \label{eq:stationnaryTheta}
    \mu_{p,A}(d\theta) = C_{p,A}\exp\Big(\chi H(\theta, p, A)\Big)d\theta,\\
    H(\theta, p, A) = v(\theta) \cdot p + \frac{\tau}{2} v(\theta) \cdot A v(\theta),
\end{align}
where $C_{p,A} >0 $ is the normalization constant, and $H(\cdot,p,A):(0,2\pi] \mapsto \dR$ is the potential associated with the drift $B(\cdot, p, A)$, since
\begin{equation*}
    \partial_\theta H(\theta,p,A) = B(\theta,p,A).
\end{equation*}

The quadratic form $ c(x) + \tau H(\theta, \nabla c(x), \nabla^2c(x))$ is the Taylor expansion of the field around the point $x$, 
\begin{equation*}
    H(\theta, \nabla c(x), \nabla^2c(x)) = \frac{1}{\tau}(c(x + \tau v(\theta)) - c(x)) + o(\tau^2).
\end{equation*}
 Depending on the terrain profile $(p,A) \in \dR^2\times M_{2,2}(\dR)$, the stationary distribution \eqref{eq:stationnaryTheta} is either uniform, unimodal, or bimodal. Hence it exhibits the desired asymptotic behavior: the uniform case corresponds to the exploration behavior, where no direction is preferred. When the distribution has a unique maximum, the unimodal case corresponds to going uphill and the bi-modal case attains its two maxima at antipodal points, corresponding to the two possible directions to walk along a crest or leave a saddle point.
Since the parameter $\tau$ represents the anticipation rate in the steering mechanism, setting $\tau = 0$ would prevent the particles from following trails, causing them to oscillate around instead (\textit{cf.} \cite{bruna2024lane}). This parameter should be taken depending on the topology of the trails in the chemotactic field.

\paragraph{Interacting particles system}
Suppose that we have a system modeling a population of $N$-ants. Let $((W^{1,i},W^{2,i}))^{1\leq i\leq N}$ a $3N$-dimensional Brownian motion, we denote by $(X^i_t,\Theta^i_t)$ the state of particle $i$ at time $t$, for $i= 1,\cdots,N$. Each is a solution to an SDE of type \eqref{sde:Antmicro} for the same pheromone field. We now introduce interaction between the particles, by coupling the chemotactic movement with the evolution of the chemical field $c^N$. We suppose that it is produced by the ants at a constant rate $\mu$. Additionally, we assume that it diffuses uniformly in the domain with diffusion coefficient $\sigma_c > 0$, and evaporates at a constant rate $\gamma > 0$. If we denote by $m^N_t$ the empirical measure of the spatial position of the ants at time $t$:
\begin{equation*}
    m^N_t = \frac{1}{N} \sum_{i=1}^N \delta_{X^i_t}.
\end{equation*}
The evolution equation for $c^N$, similarly as in the Keller-Segel model, writes as:
\begin{equation*}
    \partial_t c^N = \sigma_c \Delta c^N - \gamma c^N + \mu m^N_t.
\end{equation*}

Combining this together with the evolution equation \eqref{sde:Antmicro}, we obtain the following system of interacting particles:
\begin{equation}
    \label{sys:ParticuleSystem}
    \tag{$F^{N}_\chi$}
    \begin{cases}
        dX^{i}_t = \lambda v(\Theta^i_t)dt + \sqrt{2\sigma_x}dW^{1,i}_t,\\
        d\Theta^i_t = \chi B(\Theta^i_t, \nabla c^N(t, X^i_t), \nabla^2 c^N(t, X^i_t))dt + \sqrt{2\sigma_\theta}dW^{2,i}_t,\\
        \partial_t c^N = -\gamma c^N + \sigma_c \Delta c^N + \mu m^N,\\
        (X^i_0,\Theta^i_0) \ i.i.d. \sim \rho_0, c_{t=0} = c_0.
    \end{cases}
\end{equation}

\begin{italicsremark}
We can give sense to the stochastic parabolic equation satisfied by $c^N$ using the Duhamel formula. This notion still results in singularities at the positions of the particles. However, in this paper, we focus more on the analysis of the macroscopic model rather than the well-posedness of the particle system. Nevertheless, we will explain in Section \ref{sec:NumSim} how to avoid self-interaction singularities when simulating the particle system.

We could also assume that in the finite population system, the ants have a size $\varepsilon_N$, and that we rescale this size to zero as $N$ grows to infinity. This would lead us to consider an area around each particle where the chemo-attractant is released. Let $\varphi^N$ be a function of compact support representing the spray area depending on the scaling $\varepsilon_N$. Therefore, $\varphi^N * m^N_t$ is the production of substance $c^N$ at time $t$, where we use $*$ to denote the convolution:
\begin{equation*}
    \varphi^N * m^N_t(x) = \frac{1}{N}\sum^N_{i=1} \varphi^N(x-X^i_t).
\end{equation*}
Plugging this into the equation for $c^N$ would yield a well-possed system of moderately interacting particles. And assuming the convergence $\varphi^N \rightharpoonup \delta_0$ would yield in the limit in $N$ the same macroscopic model as the singular interacting particle system \eqref{sys:ParticuleSystem}.
\end{italicsremark}

\subsection{Our macroscopic model and a two-state extension}
It is rather natural to consider, the  Mckean-Vlasov stochastic differential equation associated to~\eqref{sys:ParticuleSystem}, representing the non-linear mean-field limit dynamic of a typical ant in a population of infinitely many ants. The system writes as follows:
\begin{equation}
    \label{sys:mckeanvlasov}
    \tag{$F^{MV}_\chi$}
    \begin{cases}
        dX_t = \lambda v(\Theta_t)dt + \sqrt{2\sigma_x}dW^1_t,\\
        d\Theta_t = \chi B(\Theta_t, \nabla c^{m}(t, X_t), \nabla^2 c^{m}(t, X_t))dt + \sqrt{2\sigma_\theta}dW^2_t,\\
        m_t \defeq \mathcal{L}(X_t),\\
        \partial_t c^{m} = -\gamma c^{m} + \sigma_c \Delta c^{m} + \mu m,\\
        (X_0, \Theta_0) \sim \rho_0, c_{t=0} = c_0.
    \end{cases}
\end{equation}

where ${m}_t$ is the marginal law of the position of an ant, and $c^{m}$ is the chemotactic field associated with $m$. We will prove an existence and uniqueness result for the previous McKean-Vlasov equation under an assumption on the regularity of the initial data (Theorem~\ref{thm:MckVExistuniq}) thanks to our fine analysis of the associated Fokker-Planck PDE.

The problem of convergence of the particle system toward the Mckean-Vlasov equation is a subject of an ongoing work. Let us now focus on the associated macroscopic PDE system.

\subsubsection*{The partial differential equation system}
From the stochastic differential equation \eqref{sys:mckeanvlasov}, we can derive the associated Fokker-Planck equation describing at the macro-scale the evolution of the joint distribution $(X,\Theta)$, denoted by $\rho : [0,T]\times \dR^2 \times [0,2\pi) \mapto \dR_+$, using Ito's formula and formally integrating by parts. Since the production of pheromone $c$ only depends on the density in the spatial $x$-variables, the production term is written as: $\mu \int \rho d\theta$. And we obtain the following system:
\begin{equation}
    \tag{$F\chi$}
    \label{sys:Formidicae}
    \begin{cases}
        \partial_t \rho = - \lambda v\cdot \nabla_x\rho - \chi \partial_\theta(B(\nabla_x c, \nabla^2_xc)\rho) + \sigma_\theta \Delta_\theta \rho + \sigma_x \Delta_x \rho \text{ on } (0,T)\times \mathbb{\dR}^2\times \mathbb{T}_{2\pi}.\\
        \partial_t c = -\gamma c + \sigma_c \Delta_x c +\mu \int \rho d\theta \text{ on }(0,T)\times \mathbb{\dR}^2,\\
        \rho_{t=0} = \rho_0, c_{t=0} = c_0,
    \end{cases}
\end{equation}

where $\mathbb{T}_{2\pi}$ is the $2\pi$-periodic Torus. It could be of modeling interest, to consider the system on a general domain for the position variable $x$, choosing some domain $\Omega \subset \dR^2$ of the plane, subject to Neumann or Dirichlet boundary conditions. We can also set the spatial domain with a periodic boundary condition, changing $\dR^2$ to $\mathbb{T}^2_1$ in the previous system \eqref{sys:Formidicae}, where $\mathbb{T}_1^2$ is the two-dimensional $1$-periodic torus. In Section \ref{sec:ProofMR}, we will prove a global in time existence and uniqueness result(see Theorem~\ref{thm:existuniqPara}), both in the case of the whole space $\dR^2$ and the two-dimensional Torus $\mathbb{T}^2_1$. Our analysis relies on the Green functions of the heat equation on these domains. We leave the general domain case for future work.

Similar to the Keller-Segel model, one could consider the parabolic-elliptic case, which models an infinitely fast-diffusing chemical. This would replace the parabolic equation for $c$ with an elliptic equation at each time $t$,

\begin{equation*}
    0 = -\gamma c_t + \sigma_c \Delta_x c_t + \mu \int \rho_t d\theta \text{ on }\dR^2.
\end{equation*}

And one could adapt the existence and uniqueness proof given in Section~\ref{sec:globalExistence} to this case.

\subsubsection*{Two-state model}

We propose an extension to the model \eqref{sys:Formidicae}, by incorporating food sources in the space domain for modeling foraging ants. Let $\alpha : \dR^2 \mapto \dR_+$ represents the concentration of food in the plane, while $\beta : \dR^2 \mapsto \dR_+$ represent the nest. Essentially, $\beta(x)$ is the instantaneous probability for an ant at $x$ to deliver food to the nest. We assume that each ant can be in one of two following states: either \textit{looking for food} (state $\alpha$), or \textit{bringing food back to the nest} (state $\beta$). Let $\rho^\alpha$ and $\rho^\beta$ represent the density of ant in state $\alpha$ and $\beta$, respectively. For a drift $B$, we denote by $\mathcal{L}_B$ the divergence form operator associated with $B$, defined as:

\begin{equation}
    \label{eq:FormDiffOpB}
    \tag{$\mathcal{L}_B$}
    \mathcal{L}_B g = \sigma_x \Delta_x g + \sigma_\theta \Delta_\theta g - \partial_\theta (Bg) - \lambda \div_x (v g).
\end{equation}

Then the two-state population will evolve according to the following system,
\begin{equation}
    \label{eq:TwostateFP}
    \begin{cases}
        \partial_t \rho^\alpha = \mathcal{L}_{B^\alpha} \rho^\alpha -\alpha \rho^{\alpha} + \beta \rho^\beta,\\
        \partial_t \rho^\beta = \mathcal{L}_{B^\beta} \rho^\beta -\beta \rho^\beta + \alpha \rho^\alpha.
    \end{cases}
\end{equation}

Each state is associated with an anticipation-reaction drift $B^\alpha$ and $B^\beta$,
and $\alpha$(resp. $\beta$) represents the instantaneous probability for an ant in state $\alpha$ to switch to state $\beta$(resp. $\beta$ to switch to state $\alpha$), depending on its spatial position. We could also suppose that changing state affects the orientation. For example, the ant immediately makes a U-turn after changing state, this would change the source term $\alpha(x) \rho^\alpha(t,x,\pi)$ in the $\beta$-equation to,
\begin{equation}
    \label{eq:TransitionTheta1}
    \alpha J[\rho_t](x,\theta) = \alpha(x)\rho(t,\theta+\pi).
\end{equation}

Or for a random change of orientation, this would lead to,
\begin{equation}
    \label{eq:TransitionTheta2}
    \alpha\Tilde{J}[\rho_t](x,\theta) = \alpha(x)\int \rho(t,x,\theta-\omega)z(\omega)d\omega,
\end{equation}
where $z$ is some probability distribution in $\mathbb{T}_{2\pi}$. Similarly, we can change the term $\beta \rho^\beta$ for $\beta J[\rho^\beta]$ in the $\alpha$-equation. For a general instantenaous orientation transition operator, let $J : L^1(\mathbb{T}_{2\pi}) \mapto L^1(\mathbb{T}_{2\pi})$, and we suppose for modeling considerations that:
\begin{enumerate}
    \item $J$ is \textit{positive}:
        \begin{equation*}
            \forall f\in L^1(\mathbb{T}_{2\pi}), f \geq 0\text{, then } J[f] \geq 0,
        \end{equation*}
    \item and \textit{mass preserving}:
        \begin{equation*}
            \int f(\theta) d\theta = \int J[f](\theta) d\theta.
        \end{equation*}
\end{enumerate}

Note that the transition operators defined in \eqref{eq:TransitionTheta1}, \eqref{eq:TransitionTheta2}, and the identity, are mass preserving and positive. We also assume that, depending on their state, the ants have a preferred direction, that we model through an additive term $D^\alpha$ (resp. $D^\beta$) in the drift,
\begin{align*}
    B^\alpha(t, \theta, x) = B(\theta,\nabla_x c^\alpha(t,x), \nabla^2_x c^\alpha(t,x)) + D^\alpha(\theta, x),\\
    B^\beta(t,\theta, x) = B(\theta,\nabla_x c^\beta(t,x), \nabla^2_x c^\beta(t,x)) + D^\beta(\theta, x),
\end{align*}

where $B$ is defined as in \eqref{eq:ThetaDrift}, where $c^\alpha$(resp. $c^\beta$) is the chemofield with which the ants in state $\alpha$ follows(resp. ants in state $\beta$). This is motivated by the following, consider the ants in state $\alpha$, we suppose that the ants can smell the food source from a distance, let the field $d^\alpha:\dR^2\mapto \dR_+$ represent this chemical information. For example, one can take it as the solution of the elliptic equation:
\begin{equation*}
    -\gamma_\alpha d^\alpha + \sigma_\alpha \Delta d^\alpha + \alpha = 0,
\end{equation*}
corresponding to the food smell being at equilibrium in the domain. This would yield the following $D^\alpha$:
\begin{equation*}
    D^\alpha(\theta, x) = \chi_\alpha v^\perp(\theta) \cdot \nabla d^\alpha(x),
\end{equation*}

for some model parameters $\chi_\alpha >0$ representing the sensitivity of an ant to the substance, $\gamma_\alpha>0$ the evaporation rate of the substance smell and $\sigma_\alpha>0$ its diffusion coefficient. For simplicity, we will assume that the food concentration remains constant over time, \textit{i.e} $D^\alpha$ and $D^\beta$ are time-homogeneous. Hence, $\beta, \alpha, D^\alpha, D^\beta$ are considered as model data. The chemo-fields $c^\alpha$ and $c^\beta$ are produced by the ants. To model their production, let $G^\alpha : (L^p_x)^2 \mapto L^p_x$ and $G^\beta : (L^p_x)^2 \mapto L^p_x$ be two functional operators. We obtain the following system:

\begin{equation}
    \label{sys:TwostateFormidicae}
    \tag{$F^{\alpha, \beta}_\chi$}
    \begin{cases}
        \partial_t \rho^\alpha = \mathcal{L}_{B^\alpha} \rho^\alpha - \alpha \rho^\alpha + \beta J[\rho^\beta],\\
        \partial_t \rho^\beta = \mathcal{L}_{B^\beta} \rho^\beta - \beta \rho^\beta + \alpha J[\rho^\alpha],\\
        \partial_t c^\alpha = \sigma_c \Delta c^\alpha - \gamma_c c^\alpha + G^\alpha[\int \rho^\alpha d\theta, \int\rho^\beta d\theta],\\
        \partial_t c^\beta = \sigma_c \Delta c^\beta - \gamma_c c^\beta + G^\beta[\int \rho^\alpha d\theta, \int\rho^\beta d\theta],
    \end{cases}
\end{equation}

where $G^\alpha[\int \rho_t^\alpha d\theta,\int \rho_t^\beta d\theta]$(resp. $G^\alpha[\int \rho_t^\alpha d\theta,\int \rho_t^\beta d\theta]$) is the production of chemoattractant $\alpha$(resp. $\beta$) at time $t$. Consideration can be given to a model where the ants interact through identical chemo-fields regardless of their states: $c^\alpha = c^\beta$.

\subsubsection*{Normalized system}
Finally, as a preliminary step of the analysis, we normalize the system \eqref{sys:Formidicae}, by introducing the following transformations:
\begin{equation*}
            \Tilde{c}(t,x) = \frac{1}{\mu}c\Big(\frac{t}{\sigma_\theta},\sqrt{\frac{\sigma_x}{\sigma_\theta}}x\Big),\ \ \Tilde{\rho}(t, x, \theta) = \rho\Big(\frac{t}{\sigma_\theta},\sqrt{\frac{\sigma_x}{\sigma_\theta}}x, \theta\Big).
\end{equation*}
Then, one can check that $\Tilde{c}$ and $\Tilde{\rho}$ are solutions to the system,
\begin{equation}
    \label{sys:NormalizedIntermediate}
    \begin{cases}
        \partial_t \Tilde{c} = -\frac{\gamma}{\sigma_\theta} \Tilde{c}  +  \frac{\sigma_c}{\sigma_x}\Delta \Tilde{c}  + \int \Tilde{\rho}(\cdot,\theta)d\theta, \\
        \partial_t \Tilde{\rho}= - \frac{\lambda}{\sqrt{\sigma_x \sigma_\theta}}v \cdot \nabla_x  \Tilde{\rho} - \frac{\chi \mu}{\sqrt{\sigma_\theta\sigma_x}} \partial_\theta(B(\nabla \Tilde{c}, \sqrt{\frac{\sigma_\theta}{\sigma_x}}\nabla^2\Tilde{c}) \Tilde{\rho}) + \Delta_\theta  \Tilde{\rho} + \Delta_x  \Tilde{\rho}.
    \end{cases}
\end{equation}
We introduce $\sigma \defeq \frac{\sigma_c}{\sigma_x}$, and we make the following substitutions, abusing notation slightly: $\chi$ replaces $\frac{\chi \mu}{\sqrt{\sigma_\theta \sigma_x}}$, $\gamma$ replaces $\frac{\gamma}{\sigma_\theta}$, $\lambda$ replaces $\lambda \sqrt{\frac{\sigma_\theta}{\sigma_x}}$ in \eqref{sys:NormalizedIntermediate}, and we change $\tau$ to $\tau \sqrt{\frac{\sigma_\theta}{\sigma_x}}$ in the definition of $B$, still denoting it by $\tau$. We then obtain in broad generality the following normalized system:
\begin{equation}
    \label{sys:FormidicaeNrmd}
    \tag{$\Bar{F}\chi$}
    \begin{cases}
        \partial_t c = -\gamma c + \sigma \Delta_x c + \int \rho d\theta, \\
        \partial_t \rho = - \lambda v\cdot \nabla_x \rho - \chi \partial_\theta(B(\nabla_x c, \nabla^2_xc)\rho) + \Delta_{\theta,x} \rho,\\
        \rho_{t=0} = \rho_0, c_{t=0} = c_0,
    \end{cases}
\end{equation}
where $\Delta_{\theta,x}$ denotes the Laplacian operator $\Delta_{\theta} + \Delta_{x}$.

%% file: AnalysisChemotaxisParabolic.tex
In this section, we present our main results. We obtain a global-in-time existence and uniqueness result to the Cauchy problems \eqref{sys:Formidicae} and \eqref{sys:TwostateFormidicae}(Theorem~\ref{thm:existuniqPara} and Theorem~\ref{thm:TwoStateFokkerPlanckMild}).
We show the propagation of the regularity of the initial data (Theorem~\ref{thm:FurtherRegularity}), and the existence and uniqueness of the Mckean-Vlasov SDE(Theorem~\ref{thm:MckVExistuniq}). After presenting some notations and preliminaries, these results are stated below.

\subsection{Notation and preliminaries}
In the following, we will study the system either on $\dR^2\times \mathbb{T}_{2\pi}$ or $\mathbb{T}^2_1\times \mathbb{T}_{2\pi}$, where $\mathbb{T}_{L} = \dR/L\mathbb{Z}$ is the $L$-periodic torus. Let $\mathcal{D}$ be the position domain, namely either $\mathcal{D} = \dR^2$ or $\mathcal{D} = \mathbb{T}^2_1$. For notation conciseness,  we will denote by $L^p_x(L^r_\theta)$  for $1\leq p,r\leq \infty$, the space $L^p(\mathcal{D}, L^r(\mathbb{T}_{2\pi}))$, equipped with its norm,
\begin{equation*}
    \|f\|_{p,r} = \left(\int_\mathcal{D} \left(\int_{\mathbb{T}_{2\pi}} |f(x,\theta)|^p d\theta\right)^\frac{p}{r}dx\right)^\frac{1}{p},
\end{equation*}
changing for $\esssup$ if $r = \infty$,
\begin{equation*}
    \|f\|_{p,\infty} = \left(\int_\mathcal{D} \left(\esssup_{\theta \in \mathbb{T}_{2\pi}} |f(x,\theta)|\right)^pdx\right)^\frac{1}{p}.
\end{equation*}

Similarly, for $T>0$, we will note $L^q_t(L^p_x(L^r_\theta))$ the Bochner spaces $L^q([0,T],L^p_x(L^r_\theta))$ and $C_t(L^p_x(L^r_\theta))$ the Banach space of continuous function $C([0,T],L^p_x(L^r_\theta))$ equipped with the $\sup$-norm. If $\mathcal{Y}$ is a Bochner space, $(\mathcal{Y})_{+}$ is the cone of non-negative functions in $\mathcal{Y}$.
$W^{1,2}_p$ will denote the anistropic Sobolev space $W^{1,2}_p([0,T], \mathcal{D})$, and $W^{2-2/p}_p$ the Besov space $W^{2-2/p}_p(\mathcal{D})$. For $k',k \in \mathbb{N}$ and $\zeta', \zeta \in (0,1)$, $C^{k' +\zeta', k+\zeta }$ denotes the H\"older space $C^{k' +\zeta', k+\zeta }([0,T],\mathcal{D}\times \mathbb{T}_{2\pi})$ or $C^{k' +\zeta', k+\zeta }([0,T],\mathcal{D})$ depending on the context. We will use $*$-symbol to represent the convolution operation, adapting to the periodic convolution in the case of the torus.

We will use multiple times the following Young convolution inequality.
\begin{proposition}(Young inequality for mixed norm)
    Let $1\leq p_i,q_i,r_i\leq \infty$, for $\phi\in L^{p_1}_x(L^{p_2}_y)$, and $\psi\in L^{p_1}_x(L^{p_2}_y)$, with
    \begin{equation*}
        1 + \frac{1}{r_i} = \frac{1}{p_i} + \frac{1}{q_i} \ \ \  i = 1,2
    \end{equation*}
    Then,
    $(\phi*\psi)\in L^{r_1}_x(L^{r_2}_y) $ and

    \begin{equation*}
        \|\phi*\psi\|_{r_1,r_2} \leq C(p,q)\|\phi\|_{p_1,p_2}\|\psi\|_{q_1,q_2}.
    \end{equation*}
\end{proposition}

We will use the following estimates on the fundamental solution of the heat equation.
\begin{proposition}
    \label{prop:funHesti}
    Let $(g_t)_{t>0}$ be the fundamental solution of the heat equation on either $\dR^2\times \mathbb{T}_{2\pi}$ or $\mathbb{T}^2_1\times \mathbb{T}_{2\pi}$.
    Then $\forall t > 0, 1 \leq p < \infty$ we have the following estimates,
    \begin{align*}
        f_p^0(t) &\defeq  \|g_t\|_{p,1}  \leq C_p\Big(1+ \frac{1}{t^\frac{p-1}{p}}\Big),\\
        f^x_{p}(t) & \defeq \max_{i=1,2}\|\partial_{x_i} g_t\|_{p,1}  \leq C_p \Big( 1+\frac{1}{t^{\frac{p-1}{p} + \frac{1}{2}}}\Big),\\
        f^\theta_{p}(t) & \defeq \|\partial_\theta g_t\|_{p,1}  \leq C_p\Big( 1+ \frac{1}{t^{\frac{p-1}{p} + \frac{1}{2}}}\Big).
    \end{align*}

    Furthermore, for $1\leq r<\frac{p}{p-1}$, we have that $f^0_p \in L^r([0,T])$ and
    \begin{equation*}
          F^0_{r,p}(T) 
        \defeq  \|f^0_p\|_{L^r([0,T])} \leq C_{p,r}\Big(T^\frac{1}{r} + T^{\frac{1}{r} - \frac{p-1}{p}}\Big),
    \end{equation*}
   Similarly, for $1\leq r<\frac{2p}{3p-2}$, we have that $ f^{x}_p, f^\theta_p  \in L^r([0,T])$,  introducing the following notations:
   \begin{equation*}
        F^x_{r,p}(T)\defeq \|f^x_p\|_{L^r([0,T])}, F^\theta_{r,p}(T) \defeq \|f^\theta_p\|_{L^r([0,T])}.
   \end{equation*}
   We have the following estimates:
    \begin{align*}
        F^x_{r,p}(T) , F^\theta_{r,p}(T) \leq C_{p,r}\Big( T^\frac{1}{r}+ T^{\frac{1}{r} - \frac{p-1}{p} + \frac{1}{2}}\Big).
    \end{align*}
\end{proposition}

The proof is given in the Appendix.

Finally, we present the following Grönwall-type lemma with delay. Although we did not find this exact form in the literature, the proof is an adaptation of classical arguments and is provided in the Appendix.

\begin{proposition}[Gr\"onwall type inequality]
    \label{prop:SGrnwllIneq}
    Let $4< p \leq \infty$, $\phi\in L^\infty_+([0,T])$ satisfies the inequality,
    \begin{equation*}
        \phi(t) \leq c_0(t) + \int^t_0 \left(\frac{ 1}{(t-s)^{\frac{1}{p} + \frac{1}{2}}} + 1\right) c_1(s) \phi(s) ds, \text{ for } a.e \ \ t\in [0,T],
    \end{equation*}
    where $c_1 \in L^p([0,T])_+$ and $c_0 \in L^\infty([0,T])_+$ is non-decreasing. With the convention that, if $p = \infty$, then $\frac{1}{p}=0$.
    
    Then,

    \begin{equation*}
        \phi(t) \leq c_0(t)M_p(\|c_1\|_{L^p(0,t)}, t), \text{ for } a.e \ \ t\in [0,T],
    \end{equation*}

    where $M_p : (\dR_+)^2 \mapto [1,\infty)$ is a positive non-decreasing continuous function only depending on $p$, and of at most exponential polynomial growth.
\end{proposition}

\subsection{Main results}
Let us start with the notion of solution we consider here.
\begin{definition}
    Let $4< p < \infty$, and $T>0$. A solution to the system \eqref{sys:FormidicaeNrmd}, is a couple $(c,\rho)$ with $c \in W^{1,2}_p$ and $\rho \in C_t(L^p_x(L^1_\theta)_+\cap L^1_{x,\theta})$, such that $c$ is a solution of the first equation in Sobolev space and $\rho$ is a mild solution to the Fokker-Planck equation, that is, it satisfies,
    \begin{equation*}
        \rho_t = \rho_0 * g_t - \chi \int _0^t (\partial_\theta g_{t-s} * (B_s\rho_s))ds - \lambda \int^t_0 (\nabla_x g_{t-s} * (v\rho_s))ds, \forall t \in [0,T],
    \end{equation*}
    where $B \in L^p_{t,x}(L^\infty_\theta)$ is defined as,
    \begin{equation*}
        B(t,\theta,x) = v^\perp(\theta) \cdot \nabla c(t,x) + \tau v^\perp(\theta) \cdot \nabla^2 c(t,x) v(\theta).
    \end{equation*}
\end{definition}

The above definition is well-possed under the given condition on $p$. 

Indeed, since the product $B\rho \in L^p_t(L^{p/2}_{x}(L^1_\theta))$, Young convolution inequality implies that for \textit{a.e} $t \in [0,T]$, $(\partial_\theta g_s * B_s\rho_s) \in L^p_x(L^1_\theta)$. And recalling the estimate on the fundamental solution,
\begin{equation*}
    \|\partial_\theta g_t\|_{\frac{p}{p-1},1} \leq C_p (1 + t^{-\frac{1}{p}-\frac{1}{2}}).
\end{equation*}

H\"older inequality in time ensures that $(\partial_\theta g_{t-\cdot} * B\rho) \in L^1_t(L^p_x(L^1_\theta))$ and the integral equation above is well defined for any $p>4$. 

As we shall see in Theorem~\ref{thm:FokkerPlanckMild}, if such $\rho$ exists then it is also a weak solution to the Fokker-Planck equation in the distributional sense. Our main result is the following theorem:
\begin{theorem}
    \label{thm:existuniqPara}
    Let $4<p< \infty$, suppose that $\gamma \geq 0, \sigma > 0, \lambda > 0$ and  $\chi \geq 0$.
    
    Then, 
    for any initial condition $\rho_0 \in L^p_x(L^1_\theta)_+ \cap L^1_{x,\theta}$, $c_0 \in W^{2-2/p}_p$,
    there exists a unique global solution $(c,\rho) \in L^p_{loc}(\dR_+,W^2_p)\times C(\dR_+, L^p_x(L^1_\theta)_+ \cap L^1_{x,\theta})$ of \eqref{sys:FormidicaeNrmd}, such that $\rho$ stays positive, its mass is preserved for all times, and it is a distributional solution to the Fokker-Planck equation.
\end{theorem}

There is no restriction on the parameter $\chi$ to prevent the blow-up in finite time of the solution. The singular attractive coupling is counterbalanced by the uniform bound on the spatial speed and by the regularization effect of the average in $\theta$, when coupling the Fokker-Planck equation in the chemotatic field equation. This is synthesized in the Averaging Lemma~\ref{lem:average}, in the form of the estimate:
\begin{equation*}
    \sup_{t\in [0,T]} \left\|\int \rho_t d\theta\right\|_p  \leq \left\|\int \rho_0 d\theta\right\|_pM_{\infty}(C_\lambda,T),
\end{equation*}
where $M_{\infty}:(\dR_+)^2\mapto \dR_+$ is a positive non-decreasing continuous function independent of $B$. This relies on the uniform ellipticity in the $x$-variable. Formally, we can derive \textit{a priori} this estimate thanks to the divergence form of the Fokker-Planck equation. Integrating with respect to $\theta$ removes the singular drift $B$, resulting in a parabolic equation for the average with a source term that can bounded. Section~\ref{sec:LinFP} addresses the previous property along with other related estimates on the linear Fokker-Planck equation for a given scalar field $B$.

From here, we obtain the following regularity results.

 \begin{theorem}
    \label{thm:FurtherRegularity}
     Suppose that $\rho_0 \in W^{2-2/p}_p\cap L^1_{x,\theta}$ and that $c_0\in C^{2+\alpha}\cap W^{2-2/p}_p$ for some $\alpha \in (0,1)$. For any $T>0$, the unique solution to the system \eqref{sys:Formidicae} is such that $c\in C^{1+\zeta/2,2+\zeta}([0,T], \mathcal{D})$ and $\rho \in W^{1,2}_p([0,T],\mathcal{D})$.
 \end{theorem}

 From this theorem, we can prove the propagation of any further H\"older regularity of the initial data. By successively iterating Schauder estimates \cite[Theorem 8.11.1, p.130]{krylov1996lectures}, using the regularity of one equation in the other. Starting from the H\"older regularity of $\partial_\theta B$ and $B$, suppose that $\rho_0$ is regular, we obtain:
\begin{equation*}
    \rho \in C^{1+\zeta,2+\zeta} \implies m \in C^{1+\zeta,2+\zeta} \implies c \in C^{2+\zeta,4+\zeta} \implies \partial_\theta B, B \in C^{\zeta,2+\zeta}.
\end{equation*}
And we can iterate up to the maximal regularity of the initial data. We formulate this property in the following Corollary.

\begin{corollary}
    \label{cor:BetterReg}
    Under the same assumption as in Theorem~\ref{thm:FurtherRegularity}, any further H\"older regularity on the initial conditions $c_0,\rho_0$ is propagated in the space variable up to any $T$, with a norm depending on the initial condition and $T$. That is, if $c_0 \in C^{2+k+\zeta}\cap W^{2-2/p}_p$ and $\rho_0 \in C^{2+k+\zeta}\cap W^{2-2/p}_p$, for $k \in \mathbb{N}\backslash \{0\}$.
    
    Then for any $T>0$,
    \begin{equation*}
        c\in C^{1+\zeta, 2+k+\zeta}([0,T], \mathcal{D}), \rho \in C^{1+\zeta, 2+k+\zeta}([0,T], \mathcal{D}\times \mathbb{T}).
    \end{equation*}
\end{corollary}

From Theorem~\ref{thm:existuniqPara} and Corollary~\ref{cor:BetterReg}, we immediately obtain the existence and uniqueness of the Mckean-Vlasov SDE, if the initial data are sufficiently smooth.
\begin{theorem}
    \label{thm:MckVExistuniq}
    Let $((\Omega,(\mathcal{F}_s)_s,\mathbb{P}),(W^{1,1},W^{1,2},W^{2}))$ be a filtered probability space with the standard hypothesis, equipped with a 3-dimensional Brownian motion. Let $c_0 \in C^{3+\zeta}(\mathcal{D})\cap W^{2-2/p}_p(\mathcal{D})$ and let $(X_0,\Theta_0)$ be $\mathcal{F}_0$-measurable random variables with a finite second moment, with joint law $\rho_0 \in \mathcal{P}(\mathcal{D}\times \mathbb{T})$, we suppose that $\rho_0$ has a density w.r.t the Lebesgue measure and that it is in the space $C^{2+\zeta}\cap W^{2-2/p}_p$.

    Then, there exists a unique strong solution to the McKean-Vlasov equation \eqref{sys:mckeanvlasov}, with initial data $(c_0,(X_0,\Theta_0))$.
\end{theorem}

As the proof is short, we give it below.
\begin{proof}[Proof of Theorem~\ref{thm:MckVExistuniq}]
     From Theorem~\ref{thm:existuniqPara}, there exists a unique solution to the PDE system~\eqref{sys:Formidicae}, further more using Corollary~\ref{cor:BetterReg} the solution has the following regularity:
    \begin{equation*}
        c \in C^{1+\zeta/2,3+\zeta}([0,T],\mathcal{D}).
    \end{equation*}
   Fixing this solution, we see that the drift $B$ associated with $c$ is Lipschitz, thus there exists a unique strong solution of the linearized version of our SDE. The time marginals solve also the linearized version of our PDE. With the same arguments as before, this linear PDE admits uniqueness. Hence the one dimensional time marginal of the law of the solution and our unique solution to the PDE are the same (as the solution of the PDE also satisfies the linearized equation). This ensures existence and uniqueness for the McKean Vlasov equation.
    
\end{proof}

Finally, we prove the global existence and uniqueness of the two-state model, under the following assumptions:

\begin{assumptionp}{$\mathcal{H}^{\alpha,\beta}_\chi$}
    \label{ass:Twostates}
    \begin{enumerate}
        \item $J:L^1((0,2\pi])\mapto L^1((0,2\pi])$, we suppose that it is a \textit{positive}, \textit{mass preserving}, and Lipschitz:
        \begin{equation*}
            \|J[\phi]-J[\psi]\|_{L^1_\theta} \leq C_J \|\psi-\psi\|_{L^1_\theta}, \forall \phi,\psi \in L^1_\theta
        \end{equation*}
        for some constant $C_J>0$. And $J[0] \equiv 0$.
        
        \item Let $G^\alpha, G^\beta: (L^p_x)^2 \mapto L^p_x$, and we suppose that there exists $C_G>0$ such that for any $f,f',g,g' \in L^p_x$:
        \begin{align*}
            \|G^\alpha[f,g]-G^\alpha[f',g']\|_{L^p_x} \leq C_G\left( \|f-f'\|_{L^p_x} + \|g-g'\|_{L^p_x} \right),\\
            \|G^\beta[f,g] - G^\beta[f',g']\|_{L^p_x} \leq C_G\left( \|f-f'\|_{L^p_x} + \|g-g'\|_{L^p_x}\right).
        \end{align*}
        and
        \begin{equation*}
            G^\beta[0,0] \equiv G^\alpha[0,0] \equiv 0.
        \end{equation*}
        \item We suppose that $\alpha,\beta : \mathcal{D}\mapto \dR_+$ are in the space $C(\mathcal{D}, \dR_+)\cap L^\infty_x$.
    \end{enumerate}
\end{assumptionp}

These assumptions allow us to obtain similar estimates as in Theorem~\ref{thm:existuniqPara} and are quite natural from a modeling perspective. The first assumption imposes some regularity on the operator that transfers the mass from one state to another, whereas the second assumption controls the production terms of the two chemotactic fields by the density of ants. The third assumption is more technical and is required to approximate $\alpha$ and $\beta$ by smooth functions.
\begin{theorem}
    \label{thm:Twostates}
    Under Assumptions \eqref{ass:Twostates}, for any initial data:
    \begin{equation*}
        (\rho^\alpha_0,\rho^\beta) \in (L^p_x(L^1_\theta)_+ \cap L^1_{x,\theta})^2,(c^\alpha_0,c^\beta_0) \in (W^{2-2/p}_p)^2,
    \end{equation*}
     there exists a unique global in-time solution of the system \eqref{sys:TwostateFormidicae}, in the space:
     \begin{equation*}
         (c^\alpha,c^\beta,\rho^\alpha,\rho^\beta) \in (L^p_{loc}(\dR_+,W^2_p))^2\times (C(\dR_+, L^p_x(L^1_\theta)_+ \cap L^1_{x,\theta}))^2.
     \end{equation*}
\end{theorem}

Before moving to our main proofs, the next section will focus on deriving estimates for the linear Fokker-Planck equation for a given scalar field $B$.

\section{Linear Fokker-Planck equation}
\label{sec:LinFP}
We here study the linear Fokker-Planck equation. To treat the two-state model, we will consider a general Fokker-Planck equation with a birth term $\eta$ and a death rate $\alpha$. But as far as the system \eqref{sys:FormidicaeNrmd} is concerned we only need to consider the case $\eta\equiv 0$, $\alpha \equiv 0$.

\begin{theorem}
    \label{thm:FokkerPlanckMild}
    Let $4<p\leq \infty$, $T>0$, and $q\geq \frac{p}{p-1}$. For any $B \in L^p_{t,x}(L^\infty_\theta)$, $\eta\in C_t(L^q_x(L^1_\theta)_+ \cap L^1_{x,\theta})$, $\alpha \in C_x \cap L^\infty_x$ and $\rho_0 \in L^q_x(L^1_\theta)_+\cap L^1_{x,\theta}$, there exists a unique $\rho \in C_t(L^q_x(L^1_\theta)_+ \cap L^1_{x,\theta})$ non-negative solution of,
    \begin{equation}
        \label{eq:generalFokkerPlanckLin}
        \rho_t = \rho_0 * g_t - \int _0^t (\partial_\theta g_{t-s} * (B_s\rho_s))ds - \int^t_0 (\nabla_x g_{t-s} * (\lambda v\rho_s)) + \int_0^t (g_{t-s} * (\eta_s + \alpha \rho_s))ds,
    \end{equation}

    with the following growth estimate,

    \begin{equation*}
        \sup_{t\in [0,T]}\|\rho_t\|_{q,1}\leq (\|\rho_0\|_{r,1} + T\sup_{[0,T]}\|\eta_s\|_{q,1})M_p(\|B\|_{p,p,\infty}+ C_\lambda+\|\alpha\|_\infty,t),
    \end{equation*}
    for some $M_p:(\dR_+)^2\mapto\dR_+$ is a positive non-decreasing continuous function, depending on $p$ only.
    Furthermore, the solution is a distributional solution of the Fokker-Planck equation, that is 
    \begin{equation}
        \label{eq:LinFokkPlanckDistributionalSol}
        \int \varphi_t \rho_t  - \int \varphi_0 \rho_0 = \int_0^t \int \Big((\partial_t \varphi + \Delta_{x,\theta}\varphi + B \partial_\theta \varphi + \lambda v \cdot \nabla_x \varphi + \alpha \varphi) \rho + \varphi \eta \Big) dx d\theta ds,
    \end{equation}
\end{theorem}
$\forall \varphi \in C^{1,2}_b, \forall t \in [0,T]$, where $C^{1,2}_b$ denotes the space of functions that are differentiable in time, twice differentiable in space, and bounded along with their derivatives. 

\begin{proof}
    Let $0<u\leq T$ to be specified later, and introduce the Banach space 
    \begin{equation*}
        E = C([0,u],L^q_x(L^1_\theta) \cap L^1_{x,\theta}),    
    \end{equation*}
    equipped with the norm,
    \begin{equation*}
        \|m\|_E = \sup_{t\in [0,u]} \|m_t\|_{q,1} + \|m_t\|_{\frac{p}{p-1},1} + \|m_t\|_{1,1}.
    \end{equation*}

    The interpolation in $L^p_x(L^1_\theta)\cap L^1_{x,\theta}$, ensures that the norm $\|\cdot\|_{\frac{p}{p-1},1}$, is finite. Let $ \Psi : E \mapto E$ be defined by, $\Psi(\nu) = \rho$, where,

    \begin{equation*}
        \rho_t = \rho_0 * g_t - \int^t_0 (\partial_\theta g_{t-s} *( B_s \nu_s)) ds - \int^t_0 (\nabla_x g_{t-s} * (\lambda v \nu_s))ds + \int_0^t (g_{t-s} * (\eta_s + \alpha \nu_s))ds.
    \end{equation*}
    This map $\Psi$ is well-defined, due to Young's inequality, as long as $p > 4$, since:
    \begin{align*}
        B\nu \in L^p_t(L^{\frac{pq}{p+q}}_x(L^1_\theta) \cap L^1_{\theta,x}) & \text{ and }   (s\mapsto \partial_\theta g_{t-s}) \in L^{\frac{p}{p-1}}_{t}(L^{\frac{p}{p-1}}_{x}(L^1_\theta) \cap L^1_{\theta,x}), \\
        v\nu \in L^\infty_t(L^q_x(L^1_\theta)\cap L^1_{x,\theta}) & \text{ and } 
        (s\mapsto \nabla_x g_{t-s}) \in L^1_{t,x,\theta},\\
        (\eta + \alpha \nu) \in L^\infty_t(L^q_x(L^1_\theta)\cap L^1_{x,\theta}) & \text{ and } 
        (s\mapsto g_{t-s}) \in L^1_{t,x,\theta}.\\
    \end{align*}

    Take any $\nu^1,\nu^2 \in E$ and note $\rho^i = \Psi(\nu^i)$ for $i = 1, 2$.
    For $r = \frac{p}{p-1}$ or $ r = q$, using Young convolution inequality for the first with,
    \begin{equation*}
        \begin{cases}
            1 + \frac{1}{r} & = \frac{p - 1}{p} + (\frac{1}{p} + \frac{1}{r})  \ \ (x),\\
            1 + 1 & = 1 + 1 \ \ (\theta),\\
        \end{cases}
    \end{equation*}

    and for the second and third terms with,
    \begin{equation*}
        \begin{cases}
            1 + \frac{1}{r} & = 1 +  \frac{1}{r}  \ \ (x),\\
            1 + 1 & = 1 + 1 \ \ (\theta).\\
        \end{cases} 
    \end{equation*}
    We obtain that,
    
    \begin{align}
        \label{est:LipFPLr}
        \|\rho^1_t - \rho^2_t\|_{r,1} & \leq \int_0^t \| \partial_\theta g_{t-s} * B_s(\nu^1_s - \nu^2_s)\|_{r,1} ds + \int^t_0 \|\nabla_x g_{t-s} * \lambda v(\nu^1_s - \nu^2_s)\|_{r,1} ds \nonumber\\
        &\ \ \ \ \  + \int^t_0 \|g_{t-s} * \alpha(\nu^1_s - \nu^2_s)\|_{r,1} ds, \nonumber\\
            & \leq  \int_0^t \| \partial_\theta g_{t-s}\|_{\frac{p}{p-1},1} \| B_s(\nu^1_s - \nu^2_s)\|_{\frac{pr}{r+p},1} ds + \int^t_0 \|\nabla_x g_{t-s} \|_{1,1} \||\lambda v|(\nu^1_s - \nu^2_s)\|_{r,1} ds\nonumber\\
            &\ \ \ \ \  + \|\alpha\|_\infty \int^t_0 \|g_{t-s}\|_{1,1} \|\nu^1_s - \nu^2_s\|_{r,1} ds, \nonumber\\
            &  \leq  \int_0^t f^\theta_{\frac{p}{p-1}}(t-s)\| B_s\|_{p,\infty} \|\nu^1_s - \nu^2_s\|_{r,1} ds + \lambda \int^t_0 f^x_{1}(t-s)\|\nu^1_s - \nu^2_s\|_{r,1} ds\nonumber\\
            &\ \ \ \ \  + \|\alpha\|_\infty \int^t_0 f_1^0(t-s) \|\nu^1_s - \nu^2_s\|_{r,1} ds, \nonumber\\
            & \leq \big(F^\theta_{\frac{p}{p-1},\frac{p}{p-1}}(u)\|B\|_{p,p,\infty} + \lambda F^x_{1,1}(u) + \|\alpha\|_\infty F^0_{1,1}(u)\big) \sup_{[0,u]}\|\nu^1_s-\nu^2_s\|_{r,1},
    \end{align}

    where we used that $f^\theta_{\frac{p}{p-1}} \in L^{\frac{p}{p-1}}([0,T])$, due to Proposition~\ref{prop:funHesti}, if 
    \begin{equation*}
        \Big(\frac{1}{p} + \frac{1}{2}\Big)\frac{p}{p-1} < 1, \text{that is, if }4 < p.
    \end{equation*}
    Now for the $L^1$ norm consider similarly, the following Young exponents,
    \begin{alignat*}{2}
        & \begin{aligned} & \begin{cases}
                1 + 1 & = 1 + (\frac{1}{p} + \frac{p-1}{p})  \ \ (x),\\
                1 + 1 & = 1 + 1 \ \ (\theta),\\
            \end{cases}
      \end{aligned}
      & \text{ and, } 
        & \begin{aligned} & \begin{cases}
                1 + 1 & = 1 + 1 \ \ (x),\\
                1 + 1 & = 1 + 1 \ \ (\theta).\\
            \end{cases} \\
      \end{aligned}
    \end{alignat*}

    So that,
    {\small
    \begin{align}
        \label{est:LipFPL1}
        \|\rho^1_t - \rho^2_t\|_{1,1} & \leq \int_0^t \| \partial_\theta g_{t-s} * B_s(\nu^1_s - \nu^2_s)\|_{1,1} ds + C_\lambda \int^t_0 \|\nabla_x g_{t-s} * \lambda v(\nu^1_s - \nu^2_s)\|_{1,1} ds, \nonumber\\
         &\hspace{2em} + \int^t_0 \|g_{t-s} * \alpha(\nu^1_s - \nu^2_s)\|_{1,1} ds, \nonumber\\
            & \leq  \int_0^t \| \partial_\theta g_{t-s}\|_{1,1} \| B_s(\nu^1_s - \nu^2_s)\|_{1,1} ds + \lambda \int^t_0 \|\nabla_x g_{t-s} \|_{1,1} \|v(\nu^1_s - \nu^2_s)\|_{1,1} ds,\nonumber\\
            &\hspace{2em}  + \|\alpha\|_\infty \int^t_0 \|g_{t-s}\|_{1,1} \|\nu^1_s - \nu^2_s\|_{1,1} ds, \nonumber\\
            &  \leq  \int_0^t f^\theta_1(t-s)\| B_s\|_{p,\infty} \|\nu^1_s - \nu^2_s\|_{\frac{p}{p-1},1} ds + \lambda \int^t_0 f^x_1(t-s)\|\nu^1_s - \nu^2_s\|_{1,1} ds,\nonumber\\
            &\hspace{2em} + \|\alpha\|_\infty \int^t_0 f_1^0(t-s) \|\nu^1_s - \nu^2_s\|_{1,1} ds, \nonumber\\
            & \leq F^\theta_{\frac{p}{p-1},1}(u) \|B\|_{p,p,\infty} \sup_{[0,u]}\|\nu^1_s-\nu^2_s\|_{\frac{p}{p-1},1}\nonumber\\
            & \hspace{4em}+ \big(\lambda F^x_{1,1}(u) + \|\alpha\|_\infty F^0_{1,1}(u)\big) \sup_{[0,u]}\|\nu^1_s-\nu^2_s\|_{1,1}.
    \end{align}}
    Here, we note that $f^\theta_1 \in L^{\frac{p}{p-1}}([0,T])$, requires that $\frac{p}{p-1} < 2$, and this is true since we already imposed that $4<p$. Combining estimates \eqref{est:LipFPLr} and \eqref{est:LipFPL1}, we obtain that for $u$ sufficiently small, $\Psi$ is a contraction. From Banach-Picard thereom, there exists a local solution. Using the continuation method we have a unique solution up to a critical time. Taking that solution, noting it $\rho$, applying analogous computations, we obtain for $r = q$ or $r = \frac{p}{p-1}$ that,
    \begin{align*}
        \|\rho_t\|_{r,1} & \leq \|\rho_0\|_{r,1} + \int_0^t \| \partial_\theta g_{t-s}\|_{\frac{p}{p-1},1} \| B_s \|_{p,\infty} \| \rho_s \|_{r,1} ds + \lambda \int^t_0 \|\nabla_x g_{t-s}\|_{1,1} \|\rho_s\|_{r,1} ds \nonumber \\
        & \ \ \ \ \ + \|\alpha\|_\infty \int^t_0 \|g_{t-s}\|_{1,1} \|\rho_s\|_{r,1} ds + \int^t_0 \|g_{t-s}\|_{1,1} \|\eta_s\|_{r,1} ds.
    \end{align*}
    From the fundamental solution estimates Proposition~\ref{prop:funHesti}, since we can bound,
    \begin{equation*}
        \max\left\{ \| \partial_\theta g_{t-s}\|_{\frac{p}{p-1},1}, \|\nabla_x g_{t-s}\|_{1,1}, \|g_{t-s}\|_{1,1}\right\}\leq C_p (1+(t-s)^{-\frac{1}{p}-\frac{1}{2}}) \ \ \forall t > s \geq 0.
    \end{equation*}

    We can apply the Gr\"onwall type inequality of Proposition~\ref{prop:SGrnwllIneq}, so that,
    \begin{equation}
        \label{est:gronwallFPLq}
        \|\rho_t\|_{r,1}  \leq \big(\|\rho_0\|_{r,1} + T\sup_{[0,T]}\|\eta_s\|_{r,1}\big)M_p\left(\|B\|_{p,p,\infty}+ (\lambda+\|\alpha\|_\infty) T^{\frac{1}{p}},t\right),
    \end{equation}
    where $M_p$ is the growth function given in Proposition~\ref{prop:SGrnwllIneq}. 
    
    And similarly for the $L^1$ norm, we obtain:
    \begin{align*}
        \|\rho_t\|_{1,1} & \leq \|\rho_0\|_{1,1} + \int_0^t \| \partial_\theta g_{t-s}\|_{1,1} \| B_s \|_{p,\infty} \| \rho_s \|_{\frac{p}{p-1},1} ds\\
        &\hspace{1cm}+ \lambda \int^t_0 \|\nabla_x g_{t-s}\|_{1,1} \|\rho_s\|_{1,1} ds\\
            & \hspace{1cm} + \|\alpha\|_\infty \int^t_0 \|g_{t-s}\|_{1,1} \|\rho_s\|_{1,1} ds + \int^t_0 \|g_{t-s}\|_{1,1} \|\eta_s\|_{1,1} ds.
    \end{align*}

    Using the bound,
    \begin{equation*}
        \max\left\{ \| \partial_\theta g_{t-s}\|_{1,1}, \|\nabla_x g_{t-s}\|_{1,1}, \|g_{t-s}\|_{1,1}\right\}\leq C_p (1+(t-s)^{-\frac{1}{2}}) \ \ \forall t > s \geq 0,
    \end{equation*}
    and Proposition~\eqref{prop:SGrnwllIneq}, we obtain for all $t\in [0,T]$,
    {\small
    \begin{equation}
        \label{est:gronwallFPL1}
        \|\rho_t\|_{1,1} \leq \Big(\|\rho_0\|_{1,1} + F^\theta_{1,1}(T)\|B\|_{p,p,\infty} \sup_{[0,t]}\|\rho_s\|_{\frac{p}{p-1},1} + T\sup_{[0,T]}\|\eta_s\|_{1,1}\Big) M_\infty(\lambda+\|\alpha\|_\infty,t).
    \end{equation}}
    Plugging the bound on the $L^{\frac{p}{p-1}}_x(L^1_\theta)$-norm from \eqref{est:gronwallFPLq} in \eqref{est:gronwallFPL1} we obtain a growth for the $L^1$ norm. These estimates prohibit finite time blow-up, thus the solution exists up to $T$.
    The proof of positivity is given in Lemma~\ref{lem:PositivityLinFokkerPlanck}.
    The fact that the mild solution is a distributional solution results from the following stability Lemma~\ref{lem:stab}, similarly as the positivity result(Lemma~\ref{lem:PositivityLinFokkerPlanck}), taking a sequence of approximated smooth solutions and controlling the convergence of the terms in $L^q_x(L^1_\theta)\bigcap L^1_{x,\theta}$ involved in the distributional formulation \eqref{eq:LinFokkPlanckDistributionalSol}.
\end{proof}

We now prove the stability result.
\begin{lemma}[Stability]
    \label{lem:stab}
    For $4<p \leq \infty$, and $\frac{p}{p-1} \leq q$, let $\rho^1, \rho^2$ two mild solutions in $C([0,T],L^q_x(L^1_\theta)\cap L^1_{x,\theta})$ of equation~\eqref{eq:generalFokkerPlanckLin}, associated respectively to initial datum $\rho^1_0, \rho^2_0 \in L^q_x(L^1_\theta)\cap L^1_{x,\theta}$, $\theta$-drifts $B^1, B^2 \in L^p_{t,x}(L^1_\theta)$, death-rate $\alpha_1,\alpha_2 \in L^\infty_x$ and source terms $\eta^1,\eta^2 \in C([0,T],L^q_x(L^1_\theta)\cap L^1_{x,\theta})$. Then, we have the following stability estimates,
    \begin{align*}
        \sup_{t \in [0,T]} \|\rho^1_s - \rho^2_s\|_{q,1} \leq & C\Big[\|\rho^1_0 - \rho^2_0\|_{q,1}  + T\sup_{[0,T]}\|\eta^1_s-\eta^2_s\|_{q,1} +\\
        &\hspace{.5cm} \|B^1-B^2\|_{p,p,\infty}F^\theta_{\frac{p}{p-1},\frac{p}{p-1}}(T)+ T\|\alpha^1-\alpha^2\|_\infty\Big],\\
         \sup_{t\in[0,T]}\|\rho^1_t-\rho^2_t\|_{1,1} \leq & C \Big[\|\rho^1_0 - \rho^2_0\|_{1,1}  + T\sup_{s\in [0,T]}\|\eta^1_s-\eta^2_s\|_{1,1} + \|B^1-B^2\|_{p,p,\infty}F^\theta_{\frac{p}{p-1},\frac{p}{p-1}}(T) +\\
        & \hspace{.5cm}  T\|\alpha^1-\alpha^2\|_\infty + F^\theta_{\frac{p}{p-1},\frac{p}{p-1}}(T) \sup_{[0,T]}\|\rho^1_s-\rho^2_s\|_{\frac{p}{p-1},1}\Big],
    \end{align*}
    
    for some positive constant $C$ depending on $T,p,\|\rho^i_0\|_{1,1},\|\rho^i_0\|_{q,1}$ ,$\|B^i\|_{p,p,\infty}$, $\|\eta^i\|_{E}$,
    $\|\alpha^i\|_\infty$ for  $i = 1,2$, and $\lambda$.
\end{lemma}

\begin{proof}
    Let $r = q$ or $r=\frac{p}{p-1}$. Using Young convolution inequality,  H\"older inequality and Gr\"onwall's Lemma, we obtain,
    \begin{align*}
        \|\rho^1_t - \rho^2_t\|_{r,1} \leq  & \|\rho^1_0 - \rho^2_0\|_{r,1}  + \lambda\int^t_0 \|\nabla_x g_{t-s}\|_{1,1}\|\rho^1_s - \rho^2_s\|_{r,1}ds \\
        &\hspace{.5cm}+ \int^t_0 \|\partial_\theta g_{t-s}\|_{\frac{p}{p-1},1}\|B^1_s\rho^1_s - B^2_s\rho^2_s\|_{\frac{pr}{p+r},1} ds\\
        & \hspace{.5cm} + \int_0^t \|g_{t-s}\|_{1,1}(\|\alpha^1\rho^1_s - \alpha^2 \rho^2_s\|_{r,1} + \|\eta^1_s-\eta^2_s\|_{r,1})ds,\\
        \leq & \|\rho^1_0 - \rho^2_0\|_{r,1}  +  \lambda\int^t_0 f^x_{1}(t-s)\|\rho^1_s - \rho^2_s\|_{r,1}ds\\
        & \hspace{.5cm} + \int^t_0 f^\theta_{\frac{p}{p-1}}(t-s)\Big(\|B^1_s - B^2_s\|_{p,\infty}\|\rho^1_s\|_{r,1} + \|B^2_s\|_{p,\infty}\|\rho^1_s - \rho^2_s\|_{r,1}\Big) ds\\
        & \hspace{.5cm} +\int_0^t (\|\alpha^1 - \alpha^2 \|_\infty \|\rho^1_s\|_{r,1} + \|\alpha^2\|_\infty \|\rho^1_s -\rho^2_s\|_{r,1} + \|\eta^1_s-\eta^2_s\|_{r,1})ds.
    \end{align*}
    Since we have the bounds,
    \begin{equation*}
        \max \left\{ f^\theta_{\frac{p}{p-1}}(t-s), f^x_1(t-s)\right\} \leq C_p (1+(t-s)^{-\frac{1}{p}-\frac{1}{2}}),
    \end{equation*}
    from Proposition~\ref{prop:SGrnwllIneq} we obtain,
    {\small
    \begin{align*}
        \|\rho^1_t - \rho^2_t\|_{r,1}  \leq & \Big[\|\rho^1_0 - \rho^2_0\|_{r,1}  + T\sup_{s\in [0,T]}\|\eta^1_s-\eta^2_s\|_{r,1} \\
        & \hspace{1.5cm}  + \big(\|B^1-B^2\|_{p,p,\infty}F^\theta_{\frac{p}{p-1},\frac{p}{p-1}}(T)+ T\|\alpha^1-\alpha^2\|_\infty\big)\sup_{s\in [0,T]}\|\rho^1_s\|_{r,1}\Big]\\
        & \hspace{1.5cm} \times M_p\left(\|B\|_{p,p,\infty}+(\lambda + \|\alpha^2\|_\infty)T^\frac{1}{p}, t\right).
    \end{align*}}
    We obtain the desired estimate from the bound on $\sup_{[0,T]}\|\rho^i_s\|_{q,1}$ of Theorem~\ref{thm:FokkerPlanckMild}.

    Similarly, for the $L^1$-norm, 

    \begin{align*}
        \|\rho^1_t - \rho^2_t\|_{1,1} \leq  & \|\rho^1_0 - \rho^2_0\|_{1,1}  + \lambda\int^t_0 \|\nabla_x g_{t-s}\|_{1,1}\|\rho^1_s - \rho^2_s\|_{1,1}ds \\
        &\hspace{.5cm}+ \int^t_0 \|\partial_\theta g_{t-s}\|_{\frac{p}{p-1},1}\|B^1_s\rho^1_s - B^2_s\rho^2_s\|_{1,1} ds\\
        & \hspace{.5cm} + \int_0^t \|g_{t-s}\|_{1,1}(\|\alpha^1\rho^1_s - \alpha^2 \rho^2_s\|_{1,1} + \|\eta^1_s-\eta^2_s\|_{1,1})ds,\\
        \leq & \|\rho^1_0 - \rho^2_0\|_{1,1}  +  \lambda\int^t_0 f^x_{1}(t-s)\|\rho^1_s - \rho^2_s\|_{1,1}ds\\
        & \hspace{.5cm} + \int^t_0 f^\theta_{\frac{p}{p-1}}(t-s)\Big(\|B^1_s - B^2_s\|_{p,\infty}\|\rho^1_s\|_{\frac{p}{p-1},1} + \|B^2_s\|_{p,\infty}\|\rho^1_s - \rho^2_s\|_{\frac{p}{p-1},1}\Big) ds\\
        & \hspace{.5cm} +\int_0^t (\|\alpha^1 - \alpha^2 \|_\infty \|\rho^1_s\|_{1,1} + \|\alpha^2\|_\infty \|\rho^1_s -\rho^2_s\|_{1,1} + \|\eta^1_s-\eta^2_s\|_{1,1})ds.
    \end{align*}
    Using again Proposition~\ref{prop:SGrnwllIneq},
    \begin{align*}
        \|\rho^1_t-\rho^2_t\|_{1,1} \leq & C \Big[\|\rho^1_0 - \rho^2_0\|_{1,1}  + T\sup_{s\in [0,T]}\|\eta^1_s-\eta^2_s\|_{1,1}\\
        & \hspace{1cm} + \|B^1-B^2\|_{p,p,\infty}F^\theta_{\frac{p}{p-1},\frac{p}{p-1}}(T) \sup_{s\in [0,T]}\|\rho^1_s\|_{\frac{p}{p-1},1}\\
        & \hspace{1cm} + T\|\alpha^1-\alpha^2\|_\infty \sup_{[0,T]}\|\rho^1\|_{1,1} + F^\theta_{ \frac{p}{p-1},\frac{p}{p-1}}(T)\|B^2\|_{p,p,\infty} \sup_{[0,T]}\|\rho^1_s-\rho^2_s\|_{\frac{p}{p-1},1}\Big],
    \end{align*}
    for some constant $C$, depending on $p$, $T$, $\|B^2\|_{p,p,\infty}$, $\|\alpha^2\|_{\infty}$ and $\lambda$. We obtain the desired estimate from the bound on $\sup_{[0,T]}\|\rho^1_s\|_{1,1}$ of Theorem~\ref{thm:FokkerPlanckMild}.
\end{proof}

\begin{lemma}[Positivity]
    \label{lem:PositivityLinFokkerPlanck}
    For $4<p \leq \infty$, and $\frac{p}{p-1} \leq q$, let $\rho$ be a mild solution in $C([0,T],L^q_x(L^1_\theta)\cap L^1_{x,\theta})$ of equation~\eqref{eq:generalFokkerPlanckLin}, associated with the initial non-negative datum $\rho_0\in (L^q_x(L^1_\theta)\cap L^1_{x,\theta})_+$, for a scalar field $B\in L^p_{t,x}(L^1_\theta)$, a death-rate $\alpha \in L^\infty_x$ and a positive birth term $\eta\in C([0,T],(L^q_x(L^1_\theta)\cap L^1_{x,\theta})_+)$. 
    
    Then, 
    \begin{equation*}
        \rho \in C([0,T],(L^q_x(L^1_\theta)\cap L^1_{x,\theta})_+),
    \end{equation*}
    that is $\rho$ stays non-negative for all times.
\end{lemma}

\begin{proof}
    Taking a sequence $B^\varepsilon \in C^\infty_b \cap L^p_{t,x}(L^\infty_\theta)$ converging in $L^p_{t,x}(L^\infty_\theta)$ to $B$, similarly $\rho^\varepsilon_0 \in C^\infty_b\cap W^{2-2/q}_q \cap W^{1}_2$ converging to $\rho_0$ in $L^q_x(L^1_\theta)$, $\eta^\varepsilon$ positive in $C^\infty_b \cap L^q_{t,x}(L^1_\theta)_+$ converging in $L^q_{t,x}(L^1_\theta)_+$ to $\eta$, and $\alpha^\varepsilon  \in C^\infty_b \cap L^\infty_x$ converging in $L^\infty_x$ to $\alpha$. There exists a unique smooth solution $\rho^\varepsilon \in C^{1,2}_b \cap W^{1,2}_2 \cap W^{1,2}_q$ (\textit{e.g} \cite[Chp 4. Thm 8. p.109]{krylov2008lectures})
    of,
    \begin{equation*}
        \begin{cases}
            \partial_t \rho^\varepsilon = \Delta \rho^\epsilon -\partial_\theta(B^\varepsilon)\rho^\varepsilon - B^\varepsilon \partial_\theta \rho^\varepsilon - \lambda v\cdot\nabla_x \rho^\varepsilon +\alpha \rho^\varepsilon + \eta^\varepsilon,\\
            \rho^\varepsilon_{t=0} = \rho^\varepsilon_0.
        \end{cases}
    \end{equation*}
    
    We use Stampacchia’s truncation method, by defining $H \in C^2(\dR)$ as,
    \begin{equation*}
        H(t) =  \begin{cases}
                    0 & \text{ if } 0 \leq t,\\
                    -\frac{t^3}{6} & \text{ if } -1 \leq t < 0,\\
                    \frac{t^2}{2} + \frac{t}{2} + \frac{1}{6} & \text{ if } t < -1.
                \end{cases}
    \end{equation*}

    $H$ and its derivatives possess the following properties:

    \begin{equation*}
        H(t) \leq t^2, \ \ -|t| \leq H'(t) \leq 0, \ \ 0 \leq H'(t)t \leq 6H(t), \ \ 0\leq H''(t) \leq 1, \ \  H''(t)t^2 \leq 6H(t), \ \ \forall t \in \dR.
    \end{equation*}

    We obtain:
    {\small
    \begin{align*}
        \frac{d}{dt}\int H(\rho^\varepsilon)dxd\theta & = \int H'(\rho^\varepsilon)(\Delta \rho^\varepsilon - \partial_\theta(B^\varepsilon \rho^\epsilon) - \lambda v\cdot \nabla_x \rho^\varepsilon + \alpha \rho^\varepsilon +\eta^\epsilon),\\
                & = -\int H''(\rho^\varepsilon) |\nabla \rho^\varepsilon|^2 + \int H'(\rho^\varepsilon)(\alpha \rho^\varepsilon +\eta^\varepsilon)+ \int H''(\rho^\varepsilon)\rho^\epsilon(B^\varepsilon \cdot \partial_\theta \rho^\varepsilon +\lambda v\cdot \nabla_x \rho^\varepsilon),
    \end{align*}}
    since $\eta^\varepsilon$ is positive, and $H'$ is non-positive, $ \int H'(\rho^\varepsilon)\eta^\varepsilon \leq 0$, this leads to,
    \begin{align*}
        \frac{d}{dt}\int H(\rho^\varepsilon)dxd\theta & \leq -\int H''(\rho^\varepsilon) |\nabla \rho^\varepsilon|^2 + \|\sqrt{H''(\rho^\varepsilon)} (B^\varepsilon  \partial_\theta \rho^\varepsilon +\lambda v\cdot \nabla_x \rho^\varepsilon)\|_2 \|\sqrt{H''(\rho^\epsilon)}\rho^\epsilon\|_2 \\
        & \hspace{1cm} + \|\alpha\|_\infty 6\int |H'(\rho^\varepsilon)\rho^\varepsilon|,\\
                & \leq -\int H''(\rho^\varepsilon) |\nabla \rho^\varepsilon|^2 + \frac{\mu}{2}\int H''(\rho^\varepsilon)(B^\varepsilon \partial_\theta \rho^\varepsilon +\lambda v\cdot \nabla_x \rho^\varepsilon)^2 \\
                &\hspace{1cm} + \frac{1}{2\mu} \int H''(\rho^\varepsilon)(\rho^\varepsilon)^2  + 6\|\alpha\|_\infty \int H(\rho^\varepsilon),\\
                & \leq -\int H''(\rho^\varepsilon) |\nabla \rho^\varepsilon|^2 +  (\lambda \vee \|B^\varepsilon\|_\infty)\frac{\mu}{2} \int H''(\rho^\varepsilon)|\nabla \rho^\varepsilon|^2 \\
                & \hspace{1cm}+ \Big(\frac{3}{\mu} + 6\|\alpha\|_\infty\Big) \int H(\rho^\varepsilon), \\
                & \leq \Big((\lambda \vee \|B^\varepsilon\|_\infty)\frac{\mu}{2} - 1\Big)\int H''(\rho^\varepsilon) |\nabla \rho^\varepsilon|^2 + \big(\frac{3}{\mu} + 6\|\alpha\|_\infty\big)\int H(\rho^\varepsilon).
    \end{align*}

    For $\mu$ sufficiently small the first term is negative, we obtain,
    \begin{equation*}
        \frac{d}{dt}\int H(\rho^\epsilon_t)dxd\theta \leq \big(\frac{3}{\mu} + 6\|\alpha\|_\infty\big)\int H(\rho^\epsilon_t).
    \end{equation*}

    Since $\rho^\varepsilon_0$ is non-negative, $\int H(\rho^\varepsilon_0) = 0$, and from Gr\"ownwall Lemma the property is preserved.
    We deduce that $\rho$ is positive, from the stability Lemma~\ref{lem:stab} that gives the convergences of the approximated solution to $\rho$.
    \end{proof}

     Finally, we state and prove the following lemma which is crucial for the analysis of the system.

    \begin{lemma}[Averaging Lemma]
        \label{lem:average}
        For $4<p \leq \infty$, and $\frac{p}{p-1} \leq q$, let $\rho$ be a solution in $C([0,T], L^q_x(L^1_\theta))$ associated with the initial data $\rho_0\in L^q_x(L^1_\theta)_+$ and with drift $B\in L^p_{t,x}(L^1_\theta)$. Then, we have the following estimate, uniformly in $B$,
        \begin{equation}
            \sup_{t\in [0,T]} \left\|\int \rho_t d\theta\right\|_p  \leq \left\|\int \rho_0 d\theta\right\|_pM_{\infty}(\lambda,T),
        \end{equation}
        where $M_{\infty}:(\dR_+)^2\mapto \dR_+$ is a positive non-decreasing continuous function independent of $B$.
    \end{lemma}
    
    \begin{proof}
        Integrating $\rho_t$ with respect to $\theta$, one gets,
        {\small \begin{align*}
            \int \rho_t d\theta & = \left(\int \rho_0 d\theta\right) *_x \left(\int g_t d\theta\right) - \int^t_0 \left(\int \partial_\theta g_{t-s}d\theta\right) *_x \left(\int B_s\rho_s d\theta\right) ds\\
            &\hspace{3em} - \int \left(\int \nabla_x g_{t-s} d\theta\right) *_x \left(\int \lambda v\rho_s d\theta \right) ds,\\
            & = \left(\int \rho_0 d\theta\right) *_x \left(\int g_t d\theta\right) - \int_0^t \left(\int \nabla_x g_{t-s} d\theta\right) *_x \left(\int \lambda v\rho_s d\theta\right) ds,
        \end{align*}}
    
        where $*_x$ indicates the convolution with respect to $(x_1,x_2)$ only, and the second inequality holds since $(\int \partial_\theta g_{t-s}d\theta) \equiv 0$. Since Theorem~\ref{thm:FokkerPlanckMild} implies that $\rho$ stays positive, we also have that,
        \begin{equation*}
            \left\|\int \lambda v\rho_s d\theta\right\|_p \leq \lambda \left\| \int |\rho_s|d\theta \right\|_p =  \lambda \left\| \int \rho_sd\theta \right\|_p.
        \end{equation*}
    
        We thus obtain,
        \begin{align*}
            \left\|\int \rho_t d\theta\right\|_p & \leq \left\|\int \rho_0 d\theta \right\|_p + \lambda \int_0^t \left\|\int \nabla g_{t-s} d\theta \right\|_{1} \left\|\int \rho_s d\theta\right\|_{p}ds,
            \\
            & \leq \left\|\int \rho_0 d\theta \right\|_p + \lambda \int_0^t f^x_{1}(t-s)\left\|\int \rho_s d\theta\right\|_{p}ds,
        \end{align*}
        using the estimate on $f^x_1$ and Proposition~\ref{prop:SGrnwllIneq}, we conclude that,
        \begin{equation*}
            \left\|\int \rho_t d\theta\right\|_p  \leq \left\|\int \rho_0 d\theta\right\|_pM_{\infty}(\lambda,t) \ \ \forall t\in [0,T].
        \end{equation*}
    \end{proof}
    
\section{Proof of the main results}
\label{sec:ProofMR}
In this section, we prove our main results announced in Section~\ref{sec:MainRes}. In Subsection~\ref{sec:globalExistence}, we present the proof of Theorem~\ref{thm:existuniqPara}. Then further regularity on the solution (see Theorem~\ref{thm:FurtherRegularity}) is obtained in Subsection~\ref{sec:FurtherRegularity}. Finally, the two-state model is tackled in Subsection~\ref{sec:Twostates}.

\subsection{Global existence}
\label{sec:globalExistence}
We now prove Theorem~\ref{thm:existuniqPara}.

\begin{proof}[Proof of Theorem~\ref{thm:existuniqPara}]
    Let $T>0$ a time to be fixed later. Introduce the Banach space 
    \begin{equation*}
        \mathcal{Y} = (W^{1,2}_p \times C_t(L^p_x(L^1_\theta))),
    \end{equation*}
    equipped with the norm
    \begin{equation*}
        \|\cdot \|_{W^{1,2}_p} + \|\cdot\|_{L^\infty_t(L^p_x(L^1_\theta))}.
    \end{equation*}

    Define the mapping $\Phi : \mathcal{Y} \mapto \mathcal{Y}$ by:
    for $(l,\nu) \in \mathcal{Y}$, $\Phi((l,\nu)) = (c,\rho)$, where $c$ is the unique solution in $W^{1,2}_p$ of
    \begin{equation*}
        \begin{cases}
            \partial_t c =\sigma \Delta c - \gamma c + \int \nu_s d\theta \text{ on } (0,T)\times \Omega,\\
            c_{t=0} = c_0,
        \end{cases}
    \end{equation*}

    and $\rho$ is the unique solution of the Fokker-Planck equation \eqref{eq:generalFokkerPlanckLin}, for $\alpha \equiv \eta \equiv 0$, and drift:

    \begin{equation*}
        B(s,x,\theta) = v^\perp(\theta)\cdot \nabla l(s,x) + \tau v^\perp(\theta)\cdot \nabla^2 l(s,x) v(\theta).
    \end{equation*}

    We note that we have the estimate,
    \begin{equation*}
        \sup_{\theta \in [0,2\pi]}|B(s,x,\theta)|\leq C_\tau (|\nabla l(s,x)| + |\nabla^2 l(s,x)|),
    \end{equation*}
    and thus,
    \begin{equation*}
        \|B\|_{p,p,\infty} \leq C_\tau \|l\|_{W^{1,2}_p}.
    \end{equation*}

    First, let $(c^1,\rho^1)$ and $(c^2,\rho^2)$ be the image of respectively $(l^1,\nu^1)$ and $(l^2,\nu^2)$, lying in $\mathcal{Y}$. From the maximal regularity result in Sobolev spaces \cite[Chp 4. Thm 8. p.109]{krylov2008lectures}, there exists $N>0$ only depending on $\gamma$, $\sigma$ and $p$ such that,
    \begin{align}
        \label{est:lipC}
        \|c^1 - c^2\|_{W^{1,2}_p} & \leq N \left\|\int (\nu^1 - \nu^2) d\theta\right\|_{p,p},\nonumber\\
        &\leq N T^{\frac{1}{p}} \sup_{t \in [0,T]}\left\|\nu^1_t - \nu^2_t \right\|_{p,1}
    \end{align}

    Note $B^1$ (resp. $B^2$) the drift associated with $l^1$ (resp. $l^2$). Since $B$ is linear w.r.t $l$, we have that,

    \begin{equation*}
        \|B^1 - B^2\|_{p,p,\infty} \leq C_\lambda \|l^1 - l^2\|_{W^{1,2}_p}.
    \end{equation*}

    Using the stability Lemma~\ref{lem:stab}, we obtain that, $\forall t \in [0,T]$,

    \begin{equation}
        \label{est:loclipRho}
        \|\rho^1_t-\rho^2_t\|_{p,1} \leq  C \|l^1 - l^2\|_{W^{1,2}_p}F^\theta_{\frac{p}{p-1},\frac{p}{p-1}}(T),
    \end{equation}

    for a $C$ depending on $p$, $\lambda$, $T$,$\|\rho_0\|_{p,1}$, $\|B^1\|_{p,p,\infty}$ and $\|B^2\|_{p,p,\infty}$. This estimate gives the local Lipschitzness of $\Phi$ and $F^\theta_{\frac{p}{p-1},\frac{p}{p-1}}$ as in Proposition~\ref{prop:funHesti}. We now prove that any iterated sequence is Cauchy, to this purpose we show that the iterated sequence of $B^n$ are uniformly bounded, this together with estimate \eqref{est:loclipRho} gives that the sequence is Cauchy for $T$ sufficiently small.
    For some $(c^0,\rho^0) \in \mathcal{Y}$, we define the sequence $((c^n,\rho^n))_{n}$ in $\mathcal{Y}$ as,
    \begin{equation*}
            (c^{n+1},\rho^{n+1}) = \Phi(c^n,\rho^n) \ \ \ \forall n \geq 0.
    \end{equation*}
    We can uniformly bound the sequence of associated drifts $B^n$.
    \begin{equation*}
        B^{n+1} = B(\theta, \nabla c^n, \nabla^2 c^n) \ \ \ \forall n \geq 0.
    \end{equation*}
    Indeed let $n \geq 2 $,
    \begin{equation*}
        \|B^{n+1}\|_{p,p,\infty} \leq C_\lambda\|c^{n}\|_{W^{1,2}_p} \leq N\left(\|c_0\|_{W^{2-2/p}_p} + \left\|\int \rho^{n-1}d\theta\right\|_{p,p}\right),
    \end{equation*}
    where $N$ is the same constant as in equation~\eqref{est:lipC}. Since for $n\geq 2$, $\rho^{n-1}$ is a solution to the Fokker-Planck equation associated with $B^{n-1}$, from the averaging Lemma~\ref{lem:average}, we obtain that,
    \begin{equation}
        \label{est:boundedBn}
        \|B^{n+1}\|_{p,p,\infty} \leq C_\lambda N\left(\|c_0\|_{W^{2-2/p}_p} + \left\|\int \rho_0 d\theta\right\|_{p}T^{\frac{1}{p}}M_\infty(\lambda ,T)\right) \ \ \ \forall n \geq 2.
    \end{equation}

    We conclude, from estimates \eqref{est:lipC}, \eqref{est:loclipRho} and \eqref{est:boundedBn}, that there exists $T$ sufficiently small so that the sequence is Cauchy.
    From the continuity of $\Phi$, the limit of the sequence is a solution to the non-linear system. Any solution will satisfy,
    \begin{equation}
        \label{est:boundB}
        \|B\|_{p,p,\infty}\leq C_\lambda N\left(\|c_0\|_{W^{2-2/p}_p} + \left\|\int \rho_0 d\theta\right\|_{p}T^{\frac{1}{p}}M_\infty(\lambda ,T)\right),
    \end{equation}
    this together with \eqref{est:loclipRho}, implies the uniqueness of the solution. By the continuation method, a unique solution exists up to a critical time, that we denote $(c,\rho)$. We already have that,
    \begin{equation}
        \label{est:croissC}
        \|c\|_{W^{1,2}_p} \leq N\left(\|c_0\|_{W^{2-2/p}_p} + \left\|\int \rho_0 d\theta\right\|_{p}T^{\frac{1}{p}}M_\infty(\lambda,T)\right).
    \end{equation}
    Thus, the growth estimate of $\|\rho_t\|_{p,1}$ from Theorem~\ref{thm:FokkerPlanckMild} together with \eqref{est:boundB}, and finally \eqref{est:croissC} prohibit finite time blow-up, thus the solution exists for all times.
\end{proof}

\subsection{Further regularity of the solution}
 \label{sec:FurtherRegularity}
 In this section, we prove that the regularity of the initial condition  propagates over time. This relies on the regularizing effects of the averaging of the density with respect to the azimuthal variable.

\begin{proof}[Proof of Theorem~\ref{thm:FurtherRegularity}]
    Take a sequence of mollified drift $B^\varepsilon$ in space and time in $C^{\infty}_b\cap L^p_{t,x}(L^\infty_\theta)$ converging to $B$ in $L^p_{t,x}(L^\infty_\theta)$, and consider the solution $\rho^\varepsilon$ of the approximated Fokker-Planck equation.
    \begin{equation}
        \label{eq:FPregEstimateReg}
        \begin{cases}
            \partial_t \rho^\varepsilon = \Delta \rho^\varepsilon - \partial_\theta(B^\varepsilon\rho^\varepsilon) -\lambda v\cdot \nabla_x \rho^\varepsilon,\\
            \rho(0,\cdot) = \rho_0.
        \end{cases}
    \end{equation}

    It is also a mild solution of the Fokker-Planck equation, so using the regularity we transfer the derivative with respect to $x$ to the solution, 

    \begin{equation*}
        \rho_t^\varepsilon = \rho_0 * g_t - \int_0^t \partial_\theta g_{t-s} * (B^\varepsilon_s\rho_s^\varepsilon) ds - \int_0^t g_{t-s} * (\lambda v\cdot \nabla_x \rho_s^\varepsilon)ds.
    \end{equation*}

    We denote by $m^\varepsilon$ the density of ants averaged over $\theta$,
    \begin{equation*}
        m^\varepsilon_t = \int \rho^\varepsilon_t d\theta.
    \end{equation*}

    Integrating equation \eqref{eq:FPregEstimateReg}, we obtain that $m^\varepsilon$ is a solution of the parabolic equation,

    \begin{equation}
        \label{eq:IntThetaFPregEstimateReg}
        \begin{cases}
            \partial_t m^\varepsilon = \Delta m^\varepsilon - \int \lambda v\cdot \nabla_x \rho^\varepsilon d\theta,\\
            m(0,\cdot) = \int \rho d\theta. 
        \end{cases}
    \end{equation}

    We can estimate the integral term as follows,
    \begin{align*}
        \int v \cdot \nabla_x \rho_t^\varepsilon d\theta & = (\int v\cdot \nabla_x \rho_0 d\theta) * g_t - \int_0^t \left(\int \partial_\theta g_{t-s} d\theta \right)* \left(\int v\cdot \nabla_x(B^\varepsilon_s\rho_s^\varepsilon)d\theta\right) ds \\ & \hspace{.5cm}- \int_0^t \int v\cdot \nabla_x g_{t-s} d\theta * \left(\int \lambda  v\cdot \nabla_x \rho_s^\varepsilon d\theta\right)ds.
    \end{align*}

    Similarly, as in the averaging lemma, we use the fact that, $ \int \partial_\theta g_{t-s} d\theta = 0$.
    We then take the $L^p$-norm, use Young convolution inequality and the Gr\"onwall type inequality Proposition~\ref{prop:SGrnwllIneq}, and obtain,
    \begin{align}
        \label{est:AverageFlowfield}
        \left\|\int v \cdot \nabla_x \rho_t^\varepsilon d\theta \right\|_p& \leq \left\|\int v\cdot \nabla_x \rho_0 d\theta\right\|_p + \lambda \int_0^t \left\|\int v\cdot \nabla_x g_{t-s} d\theta\right\|_1  \left\|\int v\cdot \nabla_x \rho_s^\varepsilon d\theta\right\|_pds, \nonumber\\
        & \leq  \left\|\int v\cdot \nabla_x \rho_0 d\theta\right\|_p + \lambda \int_0^t \left\| |\nabla_x g_{t-s}| \right\|_1  \left\|\int v\cdot \nabla_x \rho_s^\varepsilon d\theta\right\|_pds,\nonumber\\
        &\leq  \left\|\int v\cdot \nabla_x \rho_0 d\theta\right\|_p + \lambda\int_0^t  f^x_1(t-s) \left\|\int v\cdot \nabla_x \rho_s^\varepsilon d\theta\right\|_pds,\nonumber\\
        &\leq \left\|\int v\cdot \nabla_x \rho_0 d\theta\right\|_p M_\infty( \lambda,T),
    \end{align}
    where $M_\infty$ is the growth function of Proposition~\ref{prop:SGrnwllIneq}. Now, using an $L^p$ estimate \cite{krylov2008lectures} on the parabolic equation~\eqref{eq:IntThetaFPregEstimateReg}, we obtain the following bound on the norm of $m^\varepsilon$.

    \begin{equation}
        \label{est:AverageSobolev}
        \|m^\varepsilon\|_{W^{1,2}_p} \leq N\Big( \lambda\left\|\int v\cdot \nabla_x \rho^\varepsilon d\theta\right\|_{p,p} + \left\|\int\rho_0d\theta \right\|_{W^{2-2/p}_p} \Big),
    \end{equation}

    for some $N>0$ independent of $B^\varepsilon$ and $\rho_0^\varepsilon$. Combining estimates \eqref{est:AverageFlowfield} and \eqref{est:AverageSobolev}, we obtain that the sequence is uniformly bounded independently of $\varepsilon$ in $W^{1,2}_p$. This implies the existence of weak derivatives converging weakly in $L^p$ up to a subsequence, since from the stability result Lemma~\ref{lem:stab} we have the convergence in $C_t(L^p_x(L^1_\theta))$ of $\rho^\varepsilon$ to the solution and thus of $m^\varepsilon$ in $C_t(L^p_x)$, we conclude that $m\in W^{1,2}_p$.

    From Morrey's embedding theorem, for $\zeta = 1 - \frac{3}{p}$, then $m\in C^{\zeta}$. Plunging this in the equation for $c$, with initial regularity in $C^{2+\zeta}$, we obtain that $c\in C^{1+\zeta/2,2+\zeta}$. One can easily verify that from the explicit form of $B$, we have that,

    \begin{equation*}
        \partial_\theta B, B \in C^\zeta([0,T], \mathbb{T}\times \mathcal{D}).
    \end{equation*}

    Finally plugging this regularity in the Fokker-Planck equation, we obtain that \begin{equation*}
        \rho\in W^{1,2}_p([0,T],\mathcal{D}\times \mathbb{T}).
    \end{equation*}
\end{proof}
\subsection{Two-state model}
\label{sec:Twostates}
In this section we prove existence and uniqueness for the two states model \eqref{sys:TwostateFormidicae}, using similar estimates as in Theorem~\ref{thm:existuniqPara}. Throughout this section, we assume that the hypothesis in \ref{ass:Twostates} are satisfied.

We start the study with the two-states linear Fokker-Planck equation.

\begin{theorem}
    \label{thm:TwoStateFokkerPlanckMild}
    Let $4<p\leq \infty$, $T>0$, and $q\geq \frac{p}{p-1}$. Under assumption \eqref{ass:Twostates}, for any $B^\alpha,B^\beta \in L^p_{t,x}(L^\infty_\theta)$, and any initial condition $\rho^\alpha_0,\rho^\beta_0 \in L^q_x(L^1_\theta)_+\cap L^1_{x,\theta}$, there exists a unique couple $(\rho^\alpha,\rho^\beta) \in (C_t(L^q_x(L^1_\theta)_+ \cap L^1_{x,\theta}))^2$ positive solutions to the two states Fokker-Planck equation, in a mild and distributional sense, and the total mass is preserved.
    
    We also have the following stability estimate. If $(\rho^{\alpha,1},\rho^{\beta,1})$ and $(\rho^{\alpha,2},\rho^{\beta,2})$ are two solutions with same initial data, associated respectively to $(B^{\alpha,1},B^{\beta,1})$ and $(B^{\alpha,2},B^{\beta,2})$:
    \begin{equation}
        \label{est:stabTwostate}
        \sup_{t\in [0,T]} \Big(\|\rho^{\alpha,1}-\rho^{\alpha,2}\|_{q,1}+ \|\rho^{\beta,1}-\rho^{\beta,2}\|_{q,1}\Big) \leq C\Big[ F^\theta_{\frac{p}{p-1},\frac{p}{p-1}}(T)\big(\|B^{\alpha,1}-B^{\alpha,2}\|_{p,p,\infty} + \|B^{\beta,1}-B^{\beta,2}\|_{p,p,\infty}\big)\Big],
    \end{equation}
    for some $C$ depending on $\|B^{\alpha,i}\|$ for $i= 1,2$, $\|\rho^\alpha_0\|_{q,1}$, $\|\rho^\beta_0\|_{q,1}$, $\|\alpha\|_\infty$, $\|\beta\|_\infty$, $\lambda$ and $C_J$.
\end{theorem}

\begin{proof}
    Denote by $(E,\|\cdot\|_E)$ the Banach space introduced in the proof of Theorem~\ref{thm:FokkerPlanckMild}. We define $\Psi : E^2_+ \mapto E^2_+$ using Theorem~\ref{thm:FokkerPlanckMild} as,
    \begin{equation*}
        \Psi(\nu^{\alpha},\nu^{\beta}) = (\rho^{\alpha},\rho^{\beta}),
    \end{equation*}
    where $(\rho^{\alpha},\rho^{\beta})$ are the positive mild solutions of the Fokker-Planck equations:
    \begin{align}
        \begin{cases}
            \partial_t \rho^{\alpha} = \mathcal{L}_{B^\alpha} \rho^{\alpha} - \alpha \rho^{\alpha} + \alpha J [\nu^{\beta}],\\
            \partial_t \rho^{\beta} = \mathcal{L}_{B^\beta} \rho^{\beta} - \beta \rho^{\beta} + \beta J[\nu^{\alpha}],
        \end{cases}
    \end{align}

    where $\mathcal{L_{B^\alpha}}, \mathcal{L_{B^\beta}}$ are defined as in \eqref{eq:FormDiffOpB}.
    From Lemma~\ref{lem:stab} we obtain that $\Psi$ is a contraction map for $T$ sufficiently small. There exists a solution up to a critical time. Proceeding similarly as in the one-state case,
    for $r=q$ or $r =\frac{p}{p-1}$, we then have the following estimate,
    \begin{align*}
        \|\rho^{\alpha}_t\|_{r,1} & \leq \int_0^t \| \partial_\theta g_{t-s} * B^\alpha_s \rho^{\alpha}_s\|_{r,1} ds + \lambda \int^t_0 \|\nabla_x g_{t-s} * (v\rho^{\alpha}_s)\|_{r,1} ds \\
        &\hspace{.5cm}  + \int^t_0 \|g_{t-s} * \big(\alpha J[\rho^{\beta}_s]\big)\|_{r,1} ds, \\
        &  \leq  \int_0^t f^\theta_{\frac{p}{p-1}}(t-s)\| B^\alpha_s\|_{p,\infty} \|\rho^{\alpha}_s\|_{r,1} ds + \lambda \int^t_0 f^x_{1}(t-s)\|\rho^{\alpha}_s\|_{r,1} ds\\
            &\hspace{.5cm}  + \|\alpha\|_\infty C_J \int^t_0 f_1^0(t-s) \|\rho^{\beta}_s\|_{r,1} ds.
    \end{align*}
    Combining this with the same estimate on $\rho^\beta$, we prevent finite time blow-up of the solution in $L^r$, adapting the arguments of Theorem~\ref{thm:FokkerPlanckMild} we similarly obtain a growth estimate for the $L^1$-norm. We conclude that there exists a global-in-time solution of the two-state Fokker-Planck, the solution is positive, it is a distributional solution, and its total mass is preserved:
    \begin{equation*}
        \int \rho^\alpha_t + \rho^\beta_t dxd\theta = \int \rho^\alpha_0 + \rho^\beta_0 dxd\theta \hspace{.5cm}\forall t \in [0,T].
    \end{equation*}
    Finally, we prove a stability estimate w.r.t the drifts. Let $(\rho^{\alpha,1},\rho^{\beta,1})$,$(\rho^{\alpha,2},\rho^{\beta,2})$ be associated respectively with drift $(B^{\alpha,1},B^{\beta,1})$ and $(B^{\alpha,2},B^{\beta,2})$.

    {\small
    \begin{align*}
        \|\rho^{\alpha,1}_t - \rho^{\alpha,2}_t\|_{r,1} & \leq \int_0^t \| \partial_\theta g_{t-s} * (B^{\alpha,1}_s\rho^{\alpha,1} - B^{\alpha,2}_s\rho^{\alpha,2})\|_{r,1} ds + \lambda \int^t_0 \|\nabla_x g_{t-s} * v(\rho^{\alpha,1} - \rho^{\alpha,2})\|_{r,1} ds \\
        &\hspace{.5cm}  + \int^t_0 \|g_{t-s} * \alpha(J[\rho^{\beta,1}] - J[\rho^{\beta,2}])\|_{r,1} ds, \\
        &  \leq  \int_0^t f^\theta_{\frac{p}{p-1}}(t-s) \Big[\| B^{\alpha,1}_s\|_{p,\infty} \|\rho^{\alpha,1}_s - \rho^{\alpha,2}_s\|_{r,1} + \| B^{\alpha,1}_s - B^{\alpha,2}\|_{p,\infty} \|\rho^{\alpha,2}_s\|_{r,1}\Big] ds \\
        &\hspace{.5cm} + \lambda \int^t_0 f^x_{1}(t-s)\|\rho^{\alpha,1}_s - \rho^{\alpha,2}_s\|_{r,1} ds
        + \|\alpha\|_\infty L_J \int^t_0 f_1^0(t-s) \|\rho^{\beta,1}_s - \rho^{\beta,2}_s\|_{r,1} ds.
    \end{align*}}
    The same estimate holds for $\rho^{\beta,1},\rho^{\beta,2}$, so that using G\"onwall's Lemma, we obtain the desired estimate.
\end{proof}

\begin{lemma}[Averaging Lemma]
    \label{lem:averageTS}
    Let  $(\rho^\alpha,\rho^\beta) \in (C_t(L^q_x(L^1_\theta)_+ \cap L^1_{x,\theta}))^2$ be the unique solution to the two states Fokker-Planck equation, the following estimate holds for all times,

    \begin{equation*}
         \sup_{t\in [0,T]}\left(\left\| \int \rho^\alpha_t d\theta_t \right\|_p+ \left\|\int \rho^\beta_t d\theta\right\|_p \right) \leq \left(\left\|\int\rho^\alpha_0d\theta\right\|_p+\left\|\int\rho^\beta_0d\theta\right\|_p\right)M_\infty(C_\lambda + C_J,T).
    \end{equation*}
\end{lemma}

\begin{proof}
    \begin{align*}
        \left\|\int \rho^\alpha_t d\theta_t \right\|_p+ \left\|\int \rho^\beta_t d\theta\right\|_p & \leq \left\|\int\rho^\alpha_0d\theta\right\|_p+\left\|\int\rho^\beta_0d\theta\right\|_p\\
        &\hspace{.5cm} + \lambda \int_0^t\left\|\int \nabla_x g_{t-s}d\theta\right\|_1\left(\left\|\int v\rho^\alpha_s d \theta\right\|_p + \left\|\int v\rho^\beta_s d\theta\right\|_p\right)ds\\
        &\hspace{.5cm}+\int_0^t\left\|\int g_{t-s}d\theta\right\|_1\left(\|\alpha\|_\infty \left\|\int \rho^\alpha_s d \theta\right\|_p + \|\beta\|_\infty 
        \left\|\int\rho^\beta_s d\theta\right\|_p\right)ds\\
        &\hspace{.5cm}+\int_0^t\left\|\int g_{t-s}d\theta\right\|_1\left(\|\alpha\|_\infty \left\|\int J[\rho^\beta_s] d \theta\right\|_p + \|\beta\|_\infty 
        \|\int J[\rho^\alpha_s] d\theta\|_p\right)ds,\\
        & \leq \left\|\int\rho^\alpha_0d\theta\right\|_p+\left\|\int\rho^\beta_0d\theta\right\|_p\\
        &\hspace{.5cm} +\lambda \int_0^t f^x_{1,1}(t-s)\left(\left\|\int \rho^\alpha_s d\theta\right\|_p+\left\|\int \rho^\beta_s d\theta\right\|_p\right)ds\\
        &\hspace{.5cm} +C_J \int_0^t f^0_{1,1}(t-s)\left(\left\|\int \rho^\alpha_s d\theta\right\|_p+\left\|\int \rho^\beta_s d\theta\right\|_p\right)ds,\\
        &\leq \left[\left\|\int\rho^\alpha_0d\theta\right\|_p+\left\|\int\rho^\beta_0d\theta\right\|_p\right]M_\infty(\lambda + C_J,T).
    \end{align*}
\end{proof}

\begin{proof}[Proof of Theorem~\ref{thm:Twostates}]
    Applying the same strategy, we obtain local Lipschitz estimates from Schauder estimates on $c^\alpha, c^\beta$, and the stability estimate \eqref{est:stabTwostate} proved in Theorem~\ref{thm:TwoStateFokkerPlanckMild}. One then obtains the following \textit{a priori} uniform bound estimate on $B^\alpha,B^\beta$:
    \begin{align*}
        \|B^\mathfrak{a}\|_{p,p,\infty}&\leq \|c^\mathfrak{a}\|_{W^{1,2}_p},\\& \leq N \left(\|c_0^\mathfrak{a}\|_{W^{2-2/p}_p} + \left\|G^\mathfrak{a}\left[\int \rho^\alpha d\theta,\int \rho^\beta d\theta\right]\right\|_{p,p}\right),\\
        &\leq N \left(\|c_0^\mathfrak{a}\|_{W^{2-2/p}_p} +  C_G T^{1/p}\sup_{t\in [0,T]}\left[\left\|\int \rho^\alpha_t d\theta_t \right\|_p+ \left\|\int \rho^\beta_t d\theta\right\|_{p}\right]\right),
    \end{align*}
    for $\mathfrak{a} = \alpha$ or  $\mathfrak{a} = \beta$. We conclude from the two states averaging Lemma~\ref{lem:averageTS}.
\end{proof}

%% file: NumericSim.tex
In this section, we provide numerical simulations illustrating trail pattern formation. The Python source code is available on GitHub at:
\begin{center}    \texttt{https://github.com/MatthiasRakotomalala/CurvatureChemotaxis}.
\end{center}

We propose a Monte-Carlo particle simulation on the Torus, and a Finite Difference scheme for a simplified system.

\subsubsection*{Monte-Carlo particle simulation for the Mckean-Vlasov equation}

First, we provide Monte-Carlo particle simulation results for the Mckean-Vlasov equation. We refer to the thesis of M.T \cite[Chapter 7]{tomasevic2018probabilistic} for the case of the whole plane $\dR^2$, we here study the case of the Torus, this allows us to observe reinforcing trails whereas in the case of the whole space, small trails will go to infinity. As we shall see, the Torus also comes with the benefit of fast Markovian numerical methods using spectral resolution. Let $N\geq 2$ denote the number of particles. We start with the following weakly interacting particle system on the Torus:
\begin{equation}
    \label{sys:PartStrongTT}
    \begin{cases}
        dX^{i}_t = \lambda v(\Theta^i_t)dt + \sqrt{2\sigma_x}dW^{1,i}_t,\\
        d\Theta^i_t = \chi B(\Theta^i_t, \nabla c^{i,N}(t, X^i_t), \nabla^2 c^{i,N}(t, X^i_t))dt + \sqrt{2\sigma_\theta}dW^{2,i}_t,\\
        \displaystyle \partial_t c^{i,N} = -\gamma c^{i,N} + \sigma_c \Delta c^{i,N} + \frac{1}{N-1}\sum_{j=1,j\neq i}^N(\delta_{X^j} * g_{\sigma_c \varepsilon}), \text{ on } (0,T)\times \mathbb{T}^2,\\
        \text{ for }0 \leq i \leq N.
    \end{cases}
\end{equation}

Recalling that the above system represents a population of $N$ interacting particles, and is associated with the Mckean-Vlasov equation~\eqref{sys:mckeanvlasov}. In order to remove self-interaction, each particle interacts with its own chemical field $c^{i,N}$, solution to a parabolic equation with the regularised empirical measure of the other particles as a source term:
\begin{equation*}
     \frac{1}{N-1}\sum_{j=1,j\neq i}^N(\delta_{X^j} * g_{\sigma_c \varepsilon}),
\end{equation*}
where $g_{\sigma_c \varepsilon}$ is the fundamental solution to the heat equation on the Torus at time $\sigma_c \varepsilon$, for some parameter $\varepsilon$ going to zero as $N$ goes to infinity. In this paper, we do not address the problem of convergence. Instead, in this section, we assume that this particle system converges to \eqref{sys:mckeanvlasov}, implying that the error with respect to the fully interacting system is of order \(1/N\) and can be considered negligible.
That is, let $c$ be the full interaction field, solution of:
\begin{align*}
    \partial_t c = -\gamma c + \sigma_c \Delta c + \frac{1}{N}\sum^N_{j=1}(\delta_{X^j} * g_{\sigma_c \varepsilon}), \text{ on }(0,T)\times \mathbb{T}^2,
\end{align*}
then the error with the $i$-excluded field $c^{i,N}$ is:
\begin{equation*}
    \|c - c^{i,N}\|_{C^{1+\alpha/2, 2+ \alpha}} = o(1/N).
\end{equation*}
For $1\leq i \leq N$, we define $c^i$ as:
\begin{equation}
    \label{eq:simCipart}
    \partial_t c^i = -\gamma c^{i} + \sigma_c \Delta c^{i} + \frac{1}{N-1}(\delta_{X^i_t} * g_{\sigma_c \varepsilon}).
\end{equation}
So that,
\begin{equation*}
    c^{i,N} = \sum_{j=1,j\neq i}^N c^j.
\end{equation*}

We will solve equation \eqref{eq:simCipart} in the frequency domain, since in the case of the Torus the solution writes as a series. This motivates us to introduce the interacting system with Fourier truncated chemotactic fields.
For $\xi, \zeta \in \mathbb{Z}$, we denote by $c^i(t)[\xi, \zeta]$ the Fourier coefficient in the spatial variables of $c^i(t)$. For $N_F \geq 1$, we denote by $c^i_{N_F}$ the truncated series up to $N_F$:

{\small
\begin{equation*}
    c^i_{N_F}(t,x_1, x_2) = \sum_{-N_F\leq \xi, \zeta \leq N_F} \Re (c^i(t)[\xi, \zeta])\cos(2\pi (\xi x_2 + \zeta x_2))- \Im (c^i(t)[\xi, \zeta]) \sin(2\pi (\xi x_2 + \zeta x_2)),
\end{equation*}
}
where $\Re$ and $\Im$ are respectively the real and imaginary part operators. Similarly, we write $c^{i,N}_{N_F}$, the truncated series of the field $c^{i,N}$:

\begin{equation*}
    c^{i,N}_{N_F} = \sum_{j=1,j\neq i}^N c^j_{N_F}.
\end{equation*}

Using properties of the Fourier series, the derivatives of the truncated Fourier series are written as:
{\small
\begin{align*}
    \partial_{x_1}c^i_{N_F} &=  \sum_{-N_F\leq \xi, \zeta \leq N_F} -2\pi \xi \big[\Im (c^i(t)[\xi, \zeta])\cos(2\pi (\xi x_1 + \zeta x_2))+ \Re (c^i(t)[\xi, \zeta])\xi \sin(2\pi (\xi x_1 + \zeta x_2))\big],\\
    \partial_{x_2}c^i_{N_F} &=  \sum_{-N_F\leq \xi, \zeta \leq N_F} -2\pi \zeta \big[\Im (c^i(t)[\xi, \zeta])\cos(2\pi (\xi x_1 + \zeta x_2))+ \Re (c^i(t)[\xi, \zeta])\xi \sin(2\pi (\xi x_1 + \zeta x_2))\big],\\
    \partial_{x_1x_1}c^i_{N_F} &= \sum_{-N_F\leq \xi, \zeta \leq N_F} - 4\pi^2 \xi^2 \big[\Re (c^i(t)[\xi, \zeta])\cos(2\pi (\xi x_1 + \zeta x_2))- \Im (c^i(t)[\xi, \zeta]) \sin(2\pi (\xi x_1 + \zeta x_2))\big],\\
    \partial_{x_1x_2}c^i_{N_F} &= \sum_{-N_F\leq \xi, \zeta \leq N_F} - 4\pi^2 \xi \zeta \big[\Re (c^i(t)[\xi, \zeta])\cos(2\pi (\xi x_1 + \zeta x_2))- \Im (c^i(t)[\xi, \zeta]) \sin(2\pi (\xi x_1 + \zeta x_2))\big],\\
    \partial_{x_2x_2}c^i_{N_F} &= \sum_{-N_F\leq \xi, \zeta \leq N_F} - 4\pi^2 \zeta^2 \big[\Re (c^i(t)[\xi, \zeta])\cos(2\pi (\xi x_1 + \zeta x_2))- \Im (c^i(t)[\xi, \zeta]) \sin(2\pi (\xi x_1 + \zeta x_2))\big].
\end{align*}
}
From the partial differential equation \eqref{eq:simCipart} we obtain the following ordinary differential equations with random coefficients in the frequency domain. For $0\leq i \leq N, -N_F \leq \xi, \zeta \leq N_F$:
\begin{equation}
    \label{eq:simCipartFreq}
    \frac{d}{dt}c^i(t)[\xi, \zeta] = -(\gamma + \sigma_c(\xi^2 + \zeta^2))c^i(t)[\xi, \zeta] + \frac{1}{N-1}(\delta_{X^i_t} * g_{\sigma_c \varepsilon})[\xi, \zeta],
\end{equation}

with $(\delta_{X^i_t} * g_{\sigma_c \varepsilon})[\xi, \zeta]$ the Fourier coefficients of the regularised Dirac mass, they are explicitly given by:
\begin{equation*}
    (\delta_{X^i_t} * g_{\sigma_c \varepsilon})[\xi, \zeta] = \int e^{i2\pi \xi x_1 + \zeta x_2}\delta_{X^i_t}(dx) e^{-\sigma_c \varepsilon (\xi^2 + \zeta^2)} = e^{i2\pi (\xi X^{i,1}_t + \zeta X^{i,2}_t)} e^{-\sigma_c \varepsilon (\xi^2 + \zeta^2)},
\end{equation*}
with $X^i_t = (X^{i,1}_t, X^{i,2}_t)$.

Summing up, we obtain the following approximated system with truncated chemotactic fields.
\begin{equation}
    \begin{cases}
        dX^{i}_t = \lambda v(\Theta^i_t)dt + \sqrt{2\sigma_x}dW^{1,i}_t,\\
        d\Theta^i_t = \chi B(\Theta^i_t, \nabla c^{i,N}_{N_F}(t, X^i_t), \nabla^2 c^{i,N}_{N_F}(t, X^i_t))dt + \sqrt{2\sigma_\theta}dW^{2,i}_t,\\
        \frac{d}{dt}c^i[\xi, \zeta] = -(\gamma + \sigma_c(\xi^2 + \zeta^2))c^i[\xi, \zeta] + \frac{1}{N-1}(\delta_{X^i_\cdot} * g_{\sigma_c \varepsilon})[\xi, \zeta].\\
        1\leq i\leq N, -N_F \leq \xi, \zeta \leq N_F,
    \end{cases}
\end{equation}

where $\nabla c^{i,N}_{N_F}(t, X^i_t)$ and $\nabla^2 c^{i,N}_{N_F}(t, X^i_t))$ are evaluated using the coefficients $c^i(t)[\xi, \zeta]$ as described earlier. We obtain an autonomous system of $(2+ (2N_F+1)^2)N$-equations. And we then use the Euler-Maruyama forward scheme for the stochastic differential equations and the Euler implicit scheme for the ordinary differential equations on the Fourier coefficients.

It is remarkable that this simulation is linear in the number of particles $N$. One should evaluate the derivatives of $c^{i,N}$ by first summing all the fields coefficients:
\begin{equation*}
    c(t)[\xi,\zeta]=\sum_{i=1}^N c^i(t)[\xi,\zeta],
\end{equation*}
obtaining the coefficients of the total field $c(t) = \sum_i c^i(t)$, and then subtracting $c^i(t)[\xi,\zeta]$, to preserve a linear algorithm. The complexity of the simulation is:
\begin{equation*}
    \mathcal{O}(N_t N N_F^2),
\end{equation*}

where $N_t$ is the number of time steps. The simulation parameters are $\varepsilon, N_t$ and $N_F$. One should take $\varepsilon$ small enough, $N_t$ should be sufficiently large for the convergence of the Euler-Maruyama scheme, and $N_F$ sufficiently large, depending on $\sigma_c$ for the residual of the series to be negligeable.

\begin{figure}
    \centering
    \includegraphics[width=\linewidth]{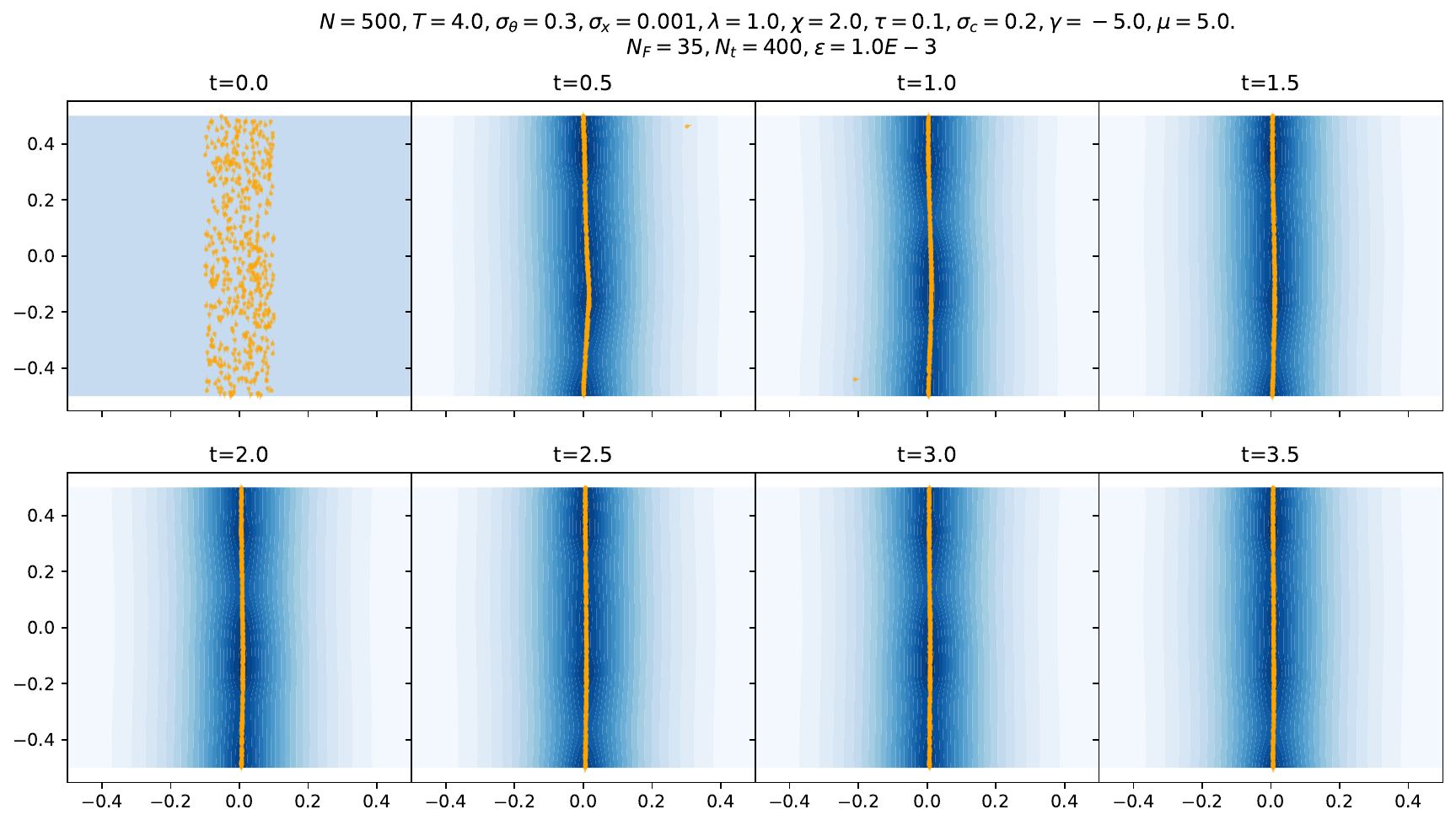}
    \caption{Monte-Carlo particle simulation for the McKean-Vlasov equation at eight different time points. The concentration of the chemotactic field is shown in \textit{blue}, while the particles are represented as points in \textit{orange}, each with an arrowhead indicating the direction of $v(\Theta^i_t)$.}
    \label{fig:MCFrSim1}
\end{figure}

\begin{figure}
    \centering
    \includegraphics[width=\linewidth]{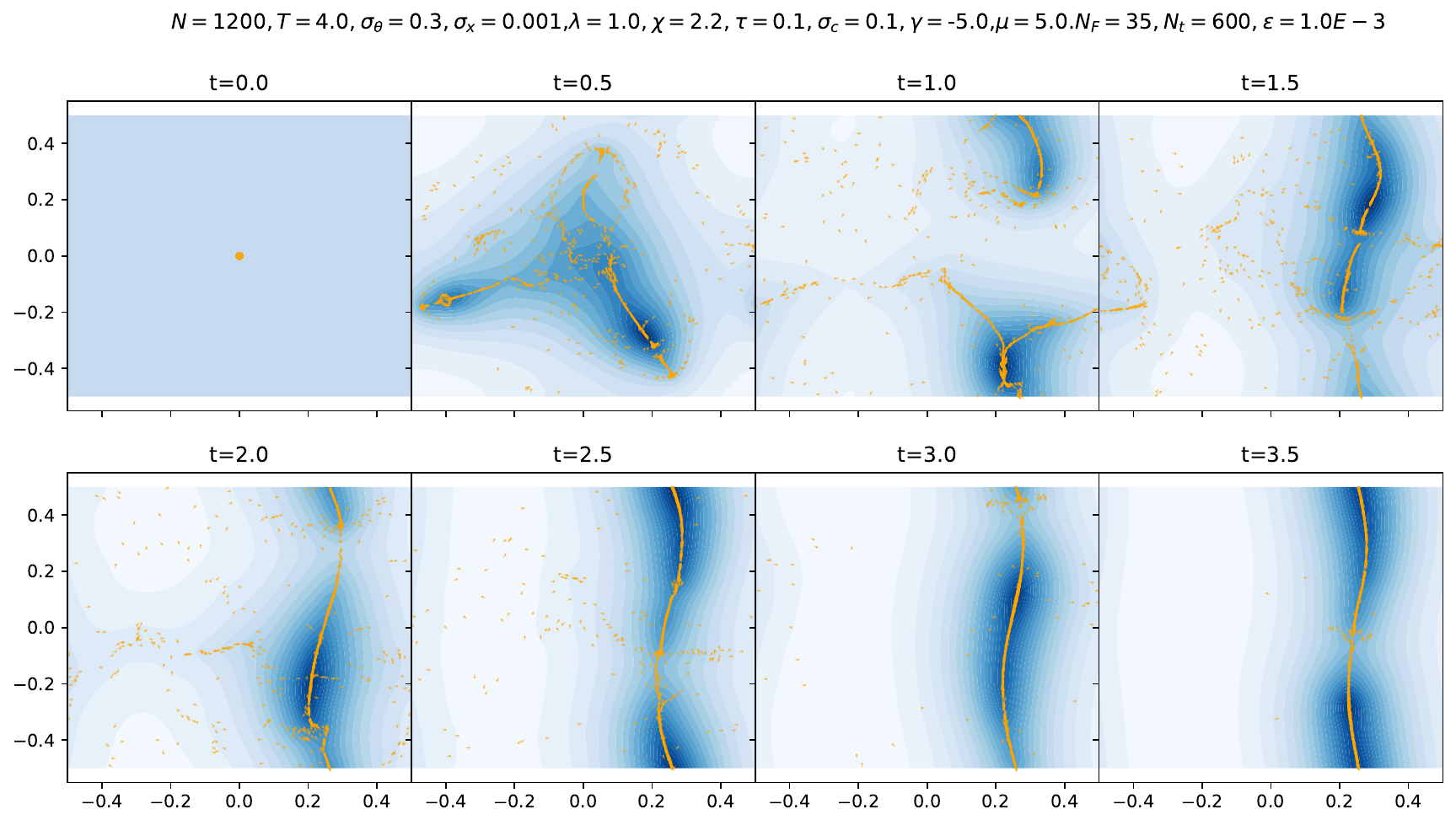}
    \caption{Monte-Carlo particle simulation for the McKean-Vlasov equation at eight different time points.}
    \label{fig:MCFrSim2}
\end{figure}

\begin{figure}
    \centering
    \includegraphics[width=\linewidth]{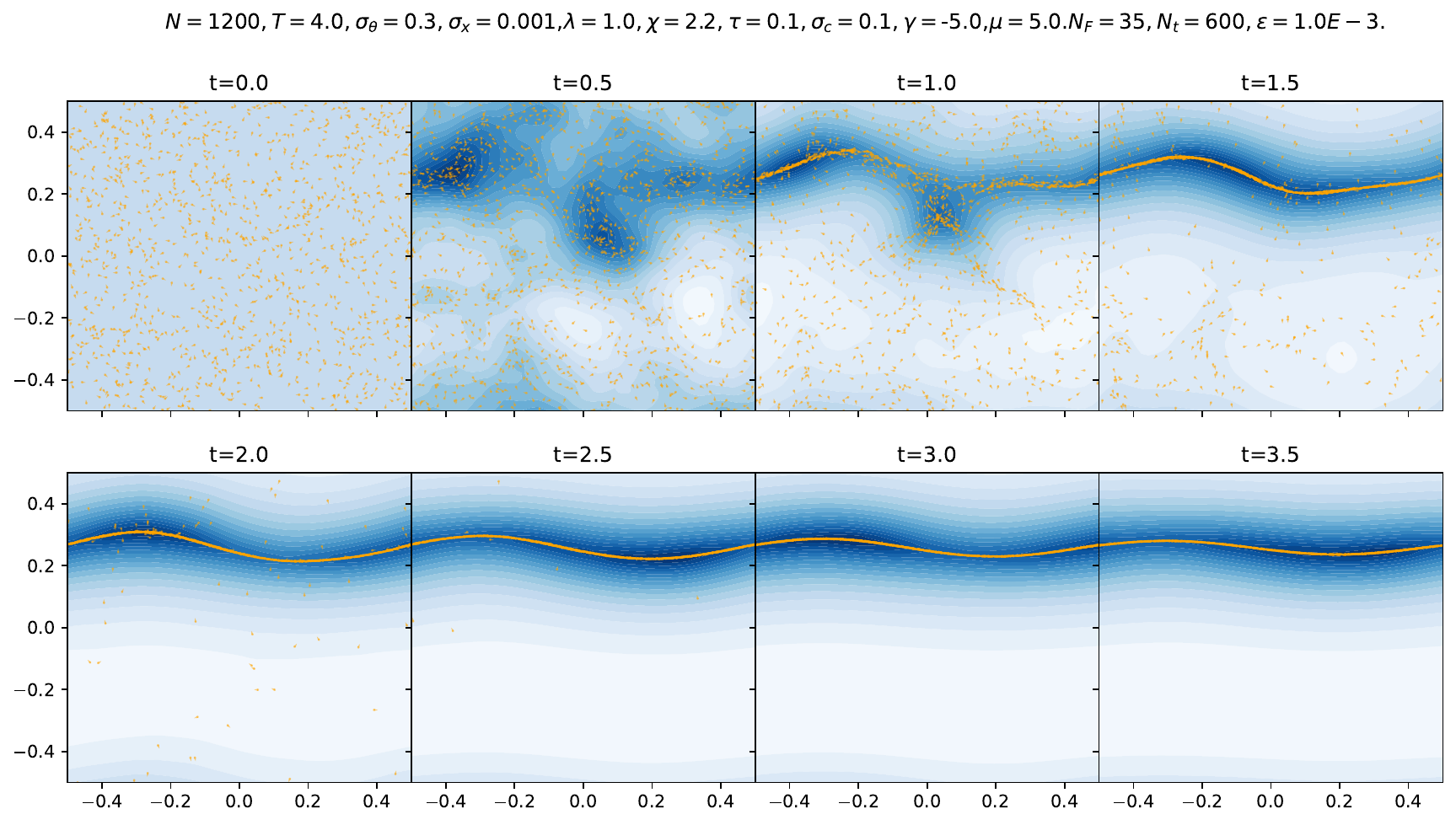}
    \caption{Monte-Carlo particle simulation for the McKean-Vlasov equation at eight different time points.}
    \label{fig:MCFrSim3}
\end{figure}

\begin{figure}
    \centering
    \includegraphics[width=\linewidth]{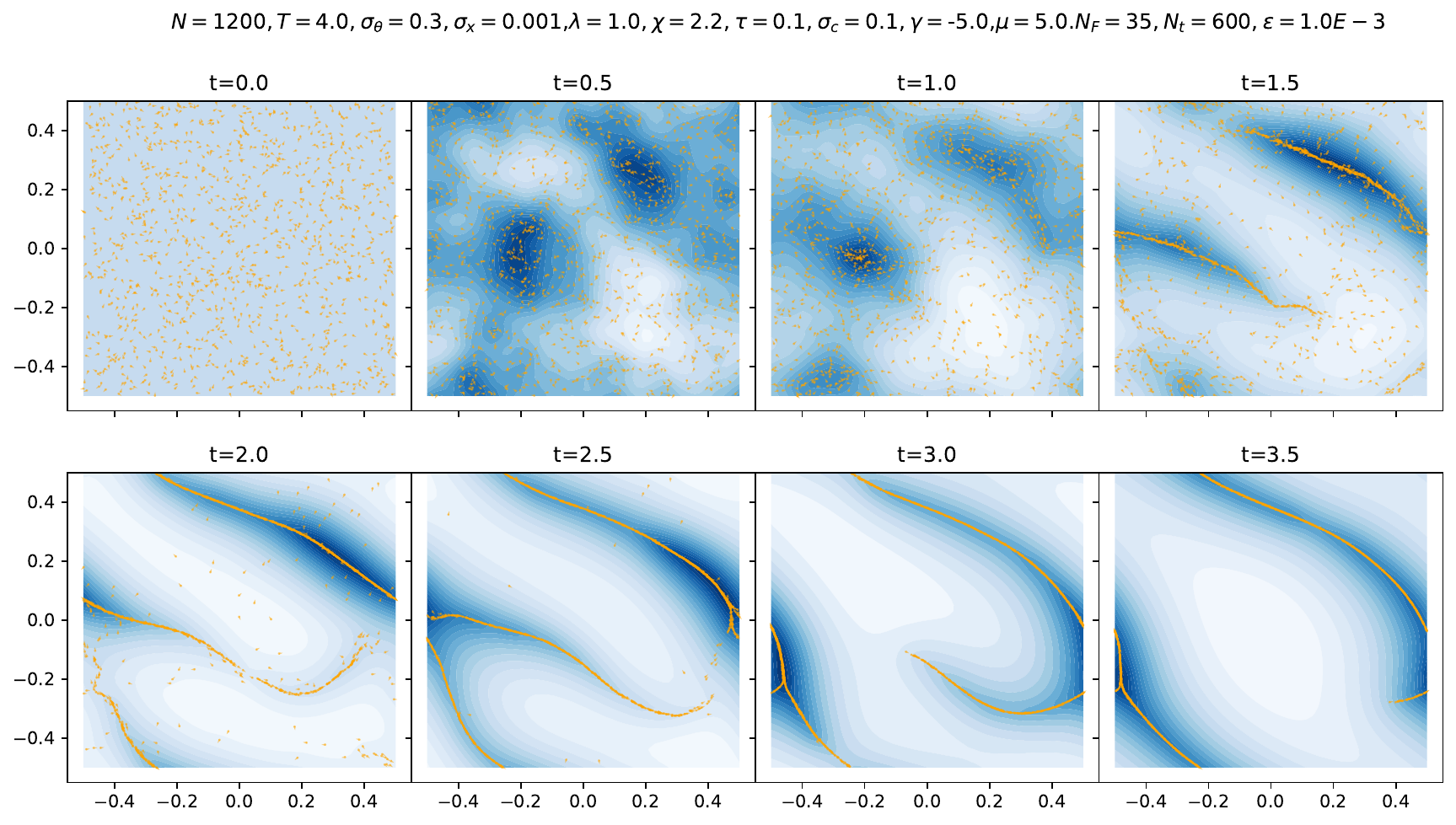}
    \caption{Monte-Carlo particle simulation for the McKean-Vlasov equation at eight different time points.}
    \label{fig:MCFrSim4}
\end{figure}

Figures~\ref{fig:MCFrSim1}, \ref{fig:MCFrSim2}, \ref{fig:MCFrSim3}, and \ref{fig:MCFrSim4} represent the evolution of four simulations obtained from the procedure described above. The concentration of the chemotactic field is shown in \textit{blue}, while the particles are represented as points in \textit{orange}, each with an arrowhead indicating the direction of $v(\Theta^i_t)$. 

We call a trail a curve along which the concentration of pheromones is locally maximal. In these simulations, we observe collective behavior, where groups of particles move along the same trail within the chemotactic field. Notice from the orientation of the arrowheads that the particles follow the trail in both directions. In a neighborhood around these trails, particles are attracted to follow these paths (see discussion in Section~\ref{sec:DerMod}). The particles are producing the chemical at their position, so if particles are aggregating on a ridge of the chemical field, they reinforce the trail. These reinforcing trail patterns are the product of this non-linear attractive mechanism, where trails attract particles, and the particles' presence further strengthens the trails.

In Figure~\ref{fig:MCFrSim1}, the initial condition is close to a trail, leading to stabilization along a single lane looping around the domain. The initial condition in Figure~\ref{fig:MCFrSim2} is a Dirac mass, while Figures~\ref{fig:MCFrSim3} and \ref{fig:MCFrSim4} are associated with uniform initial conditions. From these results, we suggest that for a certain set of parameters: the uniform solution is unstable, non-trivial stationary solutions exist and that these stationary solutions resemble trails. In the next subsection, we propose a PDE simulation of a simplified system to provide further insights into the two previous conjectures.

\subsection*{Convergence towards stationary trails using finite difference method}

As illustrated in the particle simulations, the colony is able to create trails. We then decide to study numerical solutions that are constant in the second spatial variable. We will take $\mathbb{T}_1\times \mathbb{T}_{2\pi}$ as the domain with periodic boundary conditions. Supposing that $c$ and $\rho$ are constant with respect to their second $x_2$-position variable, we obtain the following system:

\begin{equation}
    \label{sys:station1D}
    \begin{cases}
        \partial_t \rho = \sigma_\theta \partial^2_{\theta\theta}\rho +\sigma \partial^2_{xx} \rho - \chi \partial_\theta ((-\lambda \partial_x c\sin  - \partial_{xx}c \lambda \sin \cos   )\rho) - \cos \partial_x \rho \text{ in } \mathbb{T}_1\times\mathbb{T}_{2\pi},\\
        \partial_t c = -\gamma c + \partial^2_{xx}c + \int \rho d\theta \text{ in } \mathbb{T}_1.\\
        \rho_{t=0} = \rho_0, c_{t=0} = c_0.
    \end{cases}
\end{equation}
We then used a finite difference implicit scheme to solve system~\eqref{sys:station1D}. 
The numerical results are presented in Figures \ref{fig:ConvergenceFD} and \ref{fig:ConvergenceFD2}, from left to right the density integrated over $\theta$(\textit{position density}) against the time variable, the $\log$-density at the terminal time and in the last plot the position density at different times together with the concentration field of the chemical.

Recalling that if a stationary solution to system \eqref{sys:station1D} exists in dimension 2, we can extend it to be constant in the second spatial variable, and this would yield a stationary solution for the 3-dimensional system \eqref{sys:Formidicae}. We thus can observe, in the first and third graphics of Figures \ref{fig:ConvergenceFD}, \ref{fig:ConvergenceFD2} a convergence towards a distribution resembling a trail. 
Furthermore, looking at the middle plots, that is the $\log$-density at the terminal time, we can find two local maxima, at $(x=0,\theta=\frac{\pi}{2}), (x=0, \theta = \frac{3\pi}{2})$. They are at the same $x$ variable, at the top of the ridge, and are associated with the two antipodal orientations that allow to keep moving along the trail. We also note that in Figure \ref{fig:ConvergenceFD2}, the initial condition is a small perturbation of the constant solution. The solution still convergences towards a trail, providing further evidence of instability of the constant solution for the given parameters. 

\begin{figure}
    \centering
    \includegraphics[width=.9\linewidth]{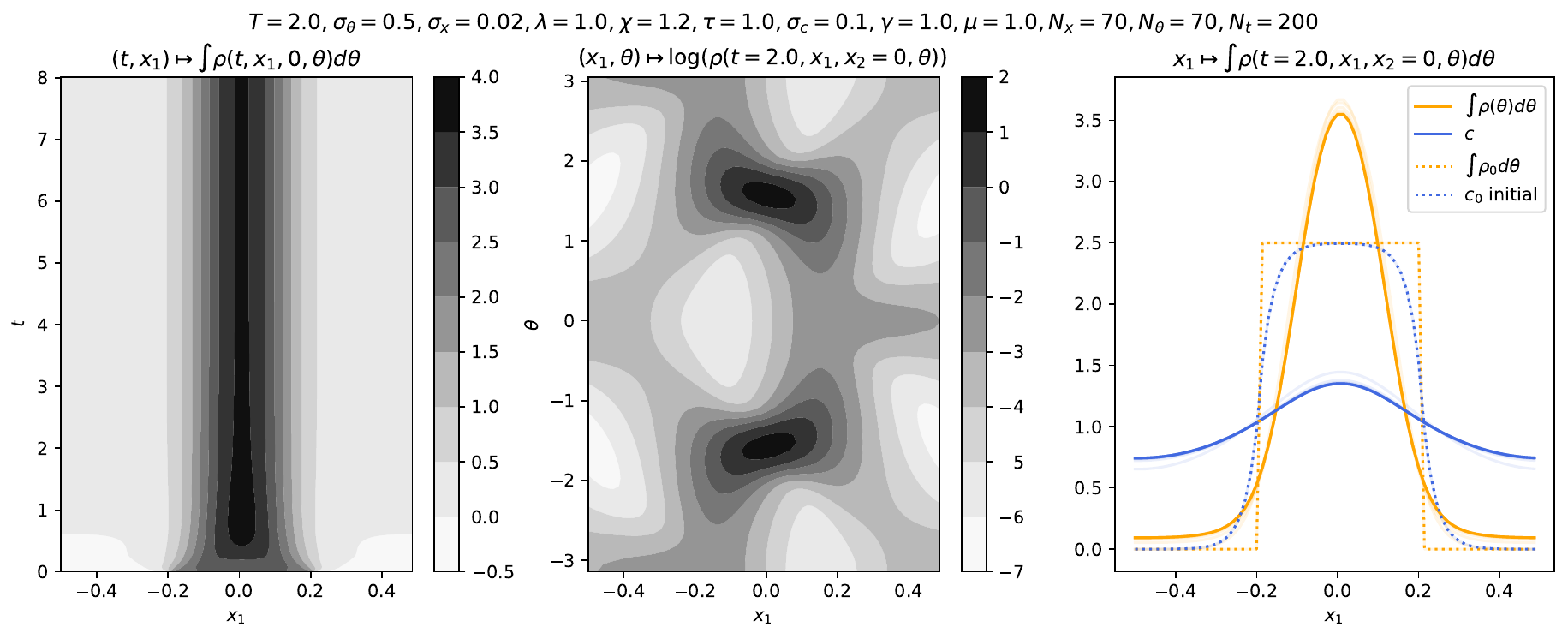}
    \caption{Finite difference for the system \eqref{sys:station1D} in 1-dimension in position space. From left to right: the $\theta$-integrated density over time, the $log$-density at the terminal time, and multiple time evaluation of the chemotatic field in \textit{blue} and the density in \textit{orange} with initial data in dotted lines.}
    \label{fig:ConvergenceFD}
\end{figure}

\begin{figure}
    \centering
    \includegraphics[width=.9\linewidth]{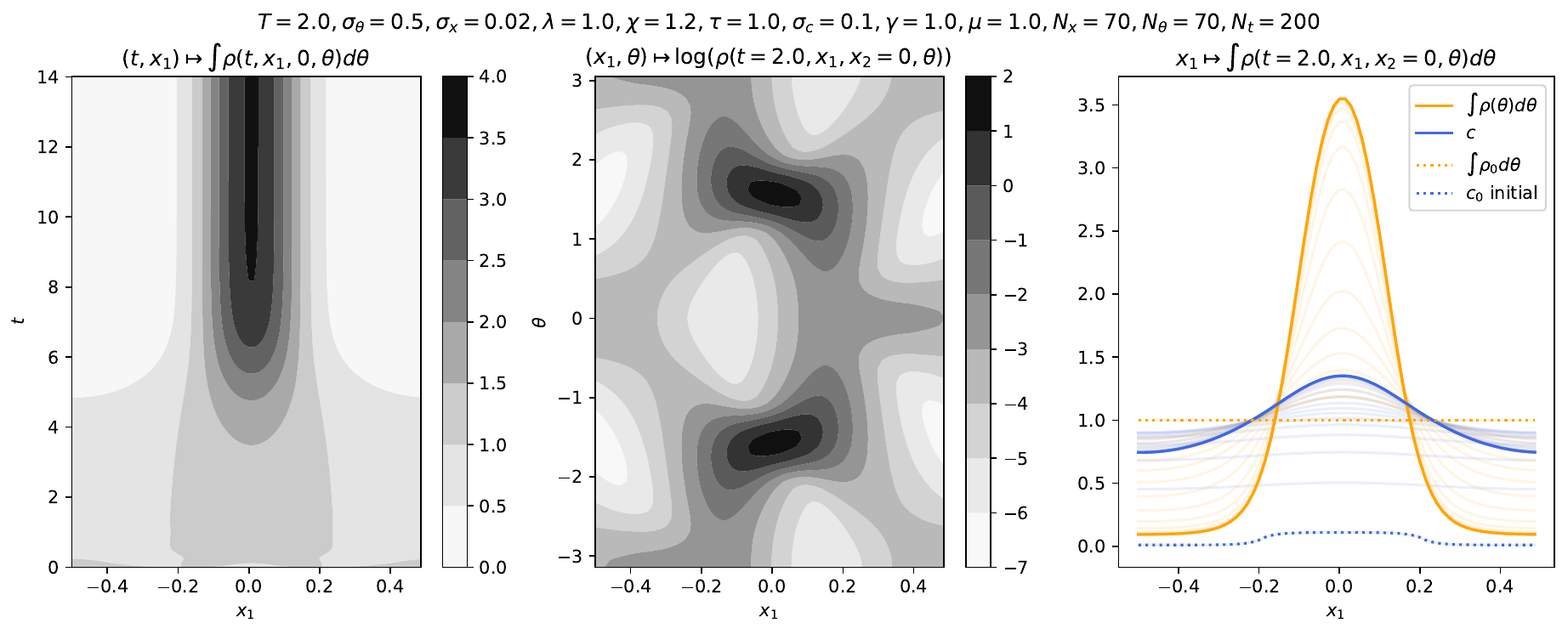}
    \caption{Instability result of the uniform distribution, for initial $\rho_0 \equiv 1/2\pi$ and $\|c_0\|_2<0.1$ small $L^2$ perturbation of the constant field, with the same parameters as in Figure \ref{fig:ConvergenceFD}.}
    \label{fig:ConvergenceFD2}
\end{figure}

\subsection*{Complex pattern and $\sigma_c$-Viscosity}
In this subsection, we present numerical evidence suggesting that, under certain conditions, more complex patterns emerge. Specifically, when the diffusion parameter $\sigma_c$ and the steering parameter $\tau$ are set to smaller values, localized pattern structures are observed. This effect is analogous to the influence of the viscosity parameter on turbulence in the Navier-Stokes equations. We begin by conducting finite difference simulations: Figures~\ref{fig:TurbulenceMC1} and \ref{fig:TurbulenceFD2}. Both simulations start from the same initial data, with all parameters left unchanged except for $\sigma_c$ and $\tau$. In the case of a large $\sigma_c$ (Figure~\ref{fig:TurbulenceFD2}), the solution to the PDE converges to a single lane, whereas with a small $\sigma_c$ (Figure~\ref{fig:TurbulenceFD1}), it converges towards two parallel trails.
\begin{figure}
    \centering
    \includegraphics[width=.9\linewidth]{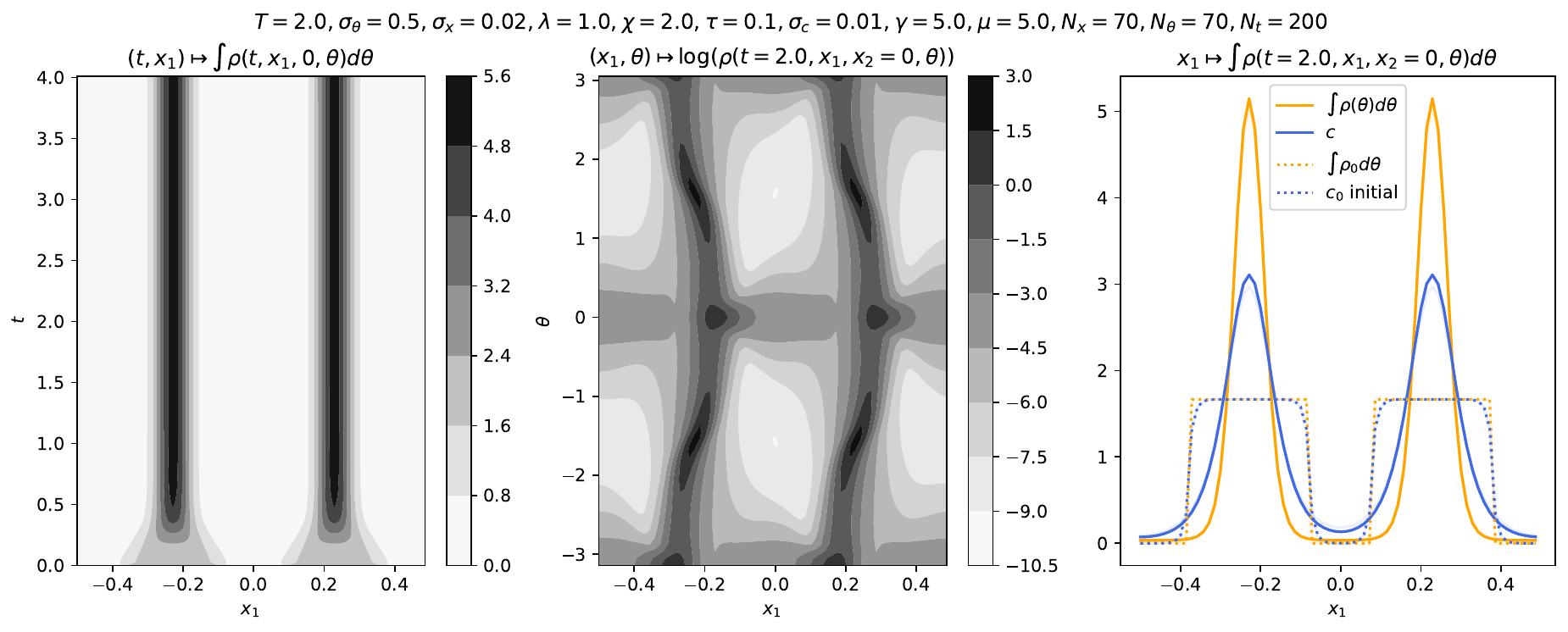}
    \caption{Finite Difference simulation in 1-dimension in space}
    \label{fig:TurbulenceFD1}
\end{figure}

\begin{figure}
    \centering
    \includegraphics[width=.9\linewidth]{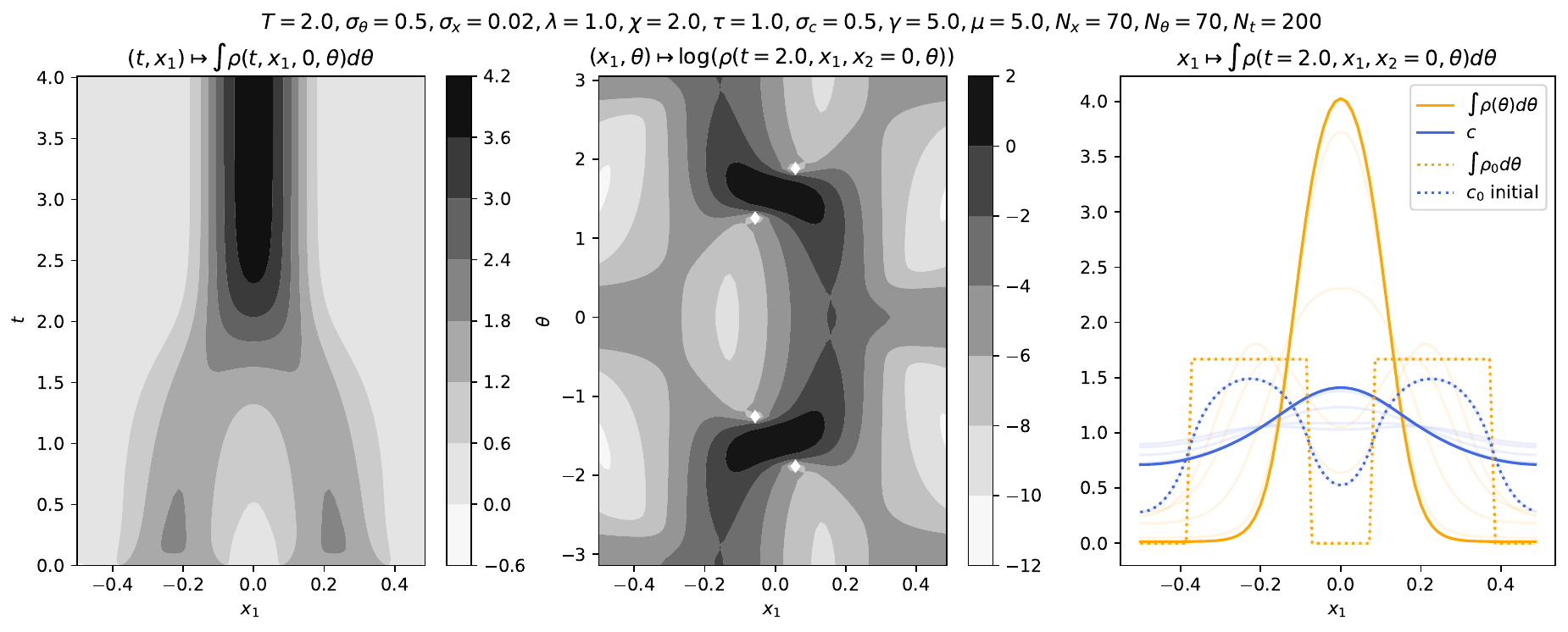}
    \caption{Finite Difference simulation in 1-dimension in space}
    \label{fig:TurbulenceFD2}
\end{figure}

We then selected parameters similar to those used in the simulation of Figure~\ref{fig:TurbulenceFD1} for the Monte-Carlo simulation, adjusting the number of Fourier coefficients $N_{Fr}$. We also vary the speed $\lambda$ to enforce local behavior. These numerical results are given in Figure~\ref{fig:TurbulenceMC1} and \ref{fig:TurbulenceMC2}.

\begin{figure}
    \centering
    \includegraphics[width=.9\linewidth]{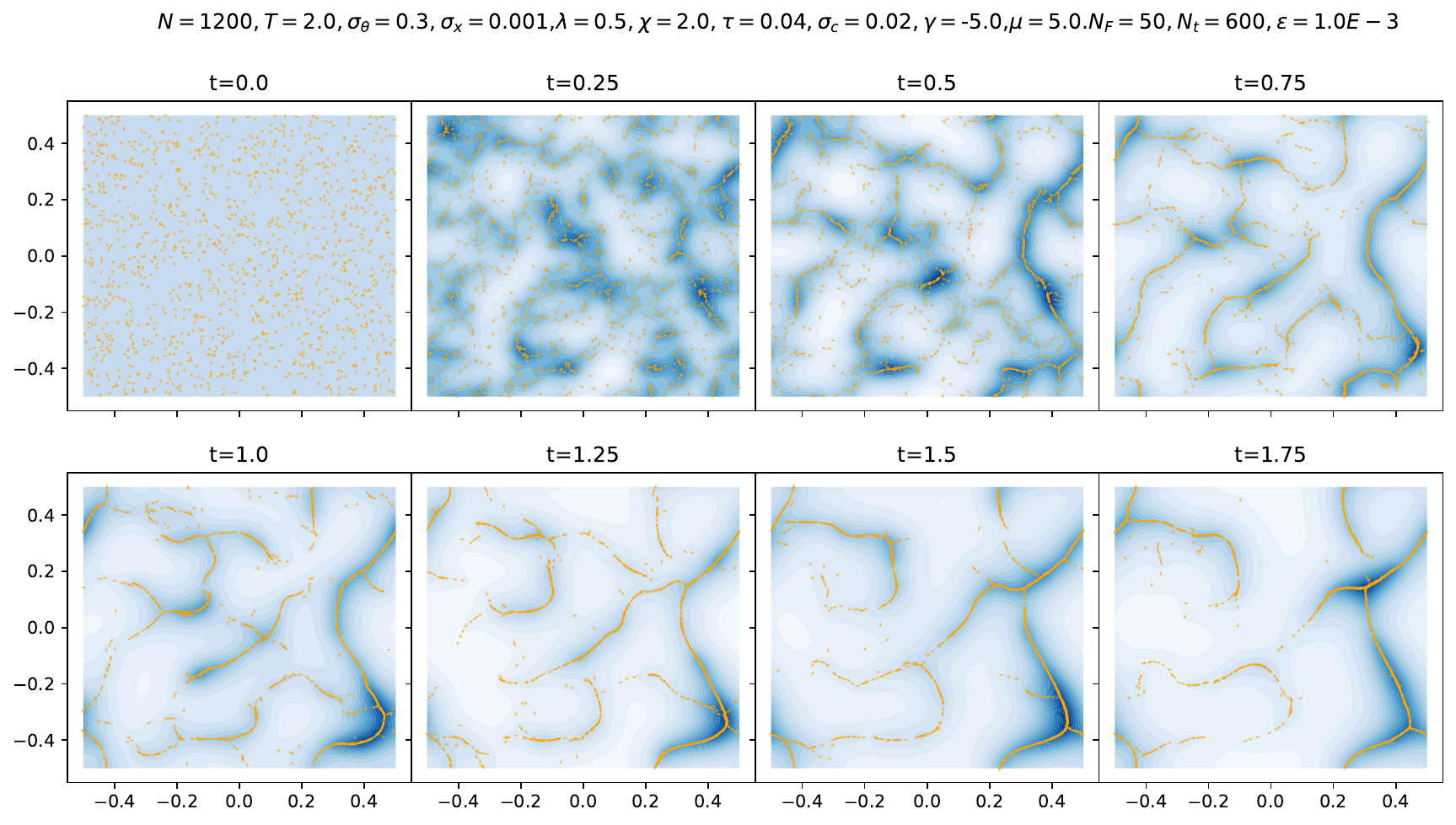}
    \caption{Monte-Carlo particle simulation for the McKean-Vlasov equation at small viscosity.}
    \label{fig:TurbulenceMC1}
\end{figure}

\begin{figure}
    \centering
    \includegraphics[width=.9\linewidth]{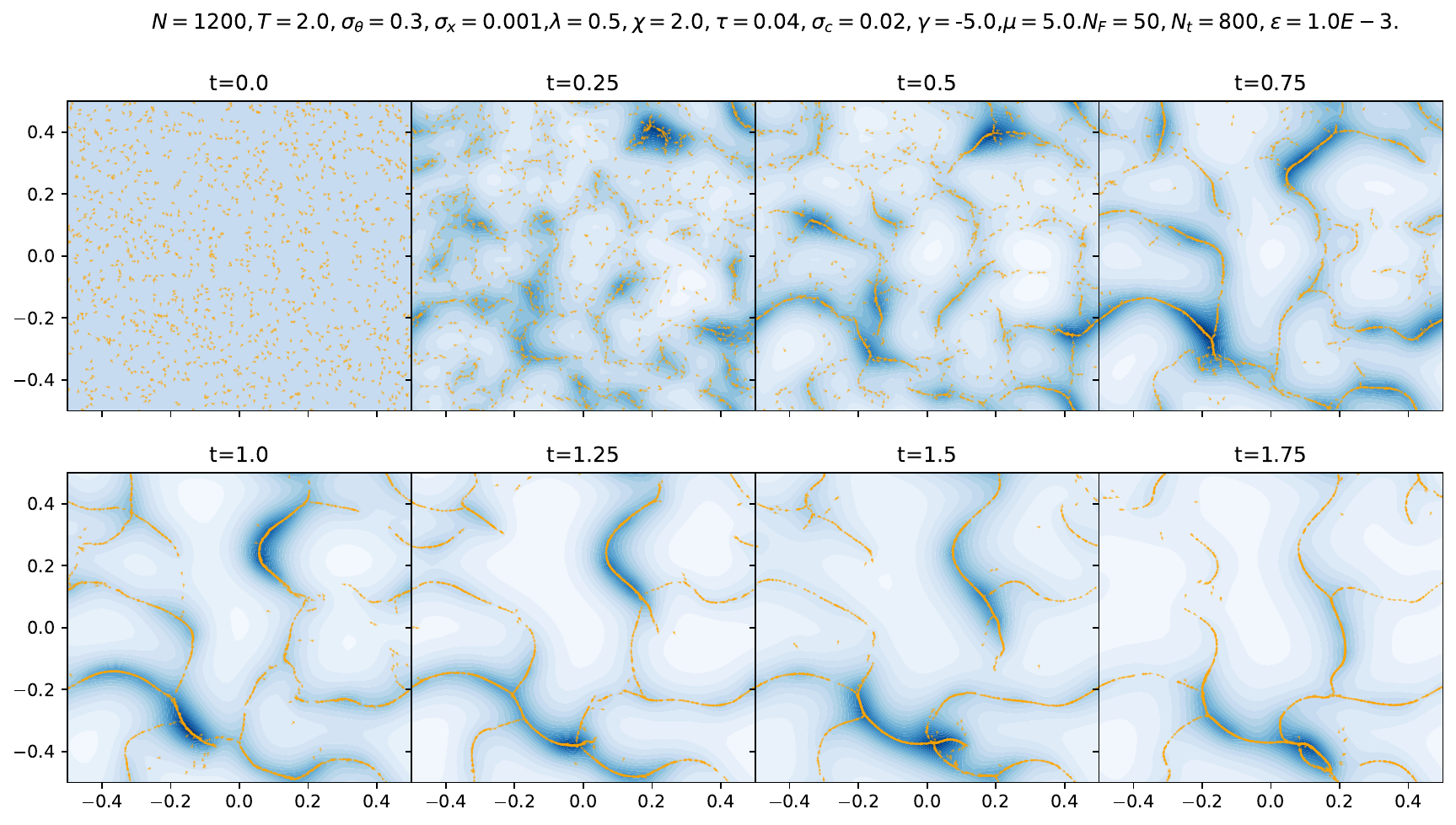}
    \caption{Monte-Carlo particle simulation for the McKean-Vlasov equation at small viscosity.}
    \label{fig:TurbulenceMC2}
\end{figure}

%% file: Appendix.tex
\begingroup
\renewcommand{\thesubsection}{\Alph{subsection}}
\renewcommand{\thetheorem}{\thesubsection.\arabic{theorem}}
\renewcommand{\thedefinition}{\thesubsection.\arabic{definition}}
\renewcommand{\thedefinitionproposition}{\thesubsection.\arabic{definitionproposition}}
\renewcommand{\theproposition}{\thesubsection.\arabic{proposition}}

\section*{Perspectives}
We believe that this model raises potentially interesting questions for both mathematical and numerical analysis, which we leave for future work.

For example, the problem of convergence and the propagation of chaos for the particle system~\eqref{sys:ParticuleSystem} to the McKean-Vlasov equation~\eqref{sys:mckeanvlasov} were not addressed in this paper.

It is also of great interest to prove the existence of non-trivial stationary solutions and to obtain asymptotic behavior properties of the model. Our numerical investigation (Section~\ref{sec:NumSim}) suggests that the model possesses non-trivial complex attractive stationary solutions, whereas the uniform solution is unstable under certain conditions. This finding relates to other open problems regarding the behavior of non-gradient flows, as the system is not associated with a natural \textit{free energy}.

For the existence and uniqueness result, the case $\sigma_x = 0$ is still unproven. Since the Mumford operator $-v \cdot \nabla_x + \Delta_\theta$ satisfies H\"ormander's condition, implying hypoelliptic regularity, one could expect some conservation of regularity for the system. Nevertheless, deriving estimates on the singularly coupled system in the case $\sigma_x = 0$ remains a non-trivial challenge.

In this paper, we provide a PDE scheme for a simplified system, assuming that the solution is constant in its second position variable, with periodic boundary conditions in the first variable, i.e., on $\mathbb{T}_1 \times \mathbb{T}_{2\pi}$ in space. However, developing a numerical scheme for the full partial differential system \eqref{sys:Formidicae} on $\mathbb{T}^2_1 \times \mathbb{T}_{2\pi}$ presents significant challenges due to the singular coupling and is left for future research.

\section*{Acknowledgement}
C.B. has been partially supported by the Chairs FDD (Institut Louis Bachelier) and FiME (Université Paris Dauphine - PSL et l’École Polytechnique).

M.T. has been partially supported by ANR SDAIM \textit{Stochastic and deterministic analysis for Irregular Models} and Chaire MMB (École Polytechnique - Muséum national d'Histoire naturelle Fondation de l'École Polytechnique - VEOLIA Environnement).


\appendix
\section*{Appendix}
\addcontentsline{toc}{section}{Appendix}

\subsection*{Fondamental solution estimates}
\begin{proof}[Proof of Proposition \ref{prop:funHesti}]
    Let us write respectively $\eta^\mathbb{T}$ and $\eta^\dR$, the fundamental solutions of the heat equation on respectively the Torus and the whole space in dimension 1, given by,
    \begin{align*}
        \eta^\dR_t(x) &= \frac{1}{\sqrt{4\pi t}} \exp\Big(-\frac{x^2}{4t}\Big),\\
        \eta^\mathbb{T}_t(x) &= \frac{1}{\sqrt{4\pi t}} \sum_{k \in \mathbb{Z}} \exp\Big(-\frac{(x + 2\pi k)^2}{4t}\Big).
    \end{align*}
    Then, $g$ is defined either, in the case $\dR^2 \times \mathbb{T}$ as,
    \begin{equation*}
        g_t(x_1,x_2,\theta) = \eta^\dR_t(x_1) \eta^\dR_t(x_2) \eta^\mathbb{T}_t(\theta),
    \end{equation*}
    or in the case $\mathbb{T}^3$ as,
    \begin{equation*}
        g_t(x_1,x_2,\theta) = \eta^\mathbb{T}_t(x_1) \eta^\mathbb{T}_t(x_2) \eta^\mathbb{T}_t(\theta).
    \end{equation*}
    Thus since it is a product, we only need to estimate the norms in the one-dimensional case and conclude using Fubini theorem. Furthermore, from the interpolation inequality in $L^\infty \cap L^1$, for any $1\leq p<\infty$,
    \begin{equation*}
        \|\eta\|_p \leq \|\eta\|_1^{\frac{1}{p}} \|\eta\|_\infty^{\frac{p-1}{p}} \forall \eta \in L^\infty \cap L^1.
    \end{equation*}
    It is classical that,
    \begin{equation*}
        \begin{cases}
            &\|\eta^\dR_t\|_1 = 1, \ \ \|\eta^\dR_t\|_\infty = \frac{1}{\sqrt{4\pi t}},\\
            &\|\partial_x\eta^\dR_t\|_1 = C_0\frac{1}{\sqrt{t}},\ \ \|\partial_x \eta^\dR_t\|_\infty = C_1\frac{1}{t}.
        \end{cases}
    \end{equation*}
    We now prove the same bounds in the case of the Torus. From the positivity, and dominated convergence Theorem,
    \begin{equation*}
        \|\eta^\mathbb{T}_t\|_1 = \int_{\mathbb{T}} \eta^\mathbb{T}_t(x) dx = \sum_{k\in \mathbb{Z}} \int_0^{2\pi}  \eta^\dR_t(x - 2\pi k) dx= \int_\dR \eta^\dR(x) dx = 1.
    \end{equation*}

    Similarly, we can bound the $L^1$-norm of the first derivative with the Gaussian from dominated convergence. After the intermediate verification of the uniform convergence of the series of the derivatives.
    \begin{equation*}
        \|\partial_x \eta^\mathbb{T}_t\|_{L^1(\mathbb{T})} \leq \|\partial_x \eta^\dR_t\|_{L^1(\dR)} = C_0 \frac{1}{\sqrt{t}}.
    \end{equation*}

    For the $L^\infty$-norm estimate, we will use the Fourier representation of $\eta^\mathbb{T}$, that is,
    \begin{equation*}
        \eta^{\mathbb{T}}_t(x) = 1 + 2\sum_{n\geq 1} e^{-tn^2}cos(nx).
    \end{equation*}

    It is straightforward to check that,
    \begin{equation*}
        0 \leq \eta^{\mathbb{T}}_t(x) = 1 + 2\sum_{n\geq 1} e^{-tn^2} \leq 1 + 2\int^\infty_0 e^{-ty^2}dy = 1 + \sqrt{\frac{\pi}{t}}.
    \end{equation*}

    Finally, differentiating inside the sum, from uniform convergence, we obtain,
    \begin{equation*}
        \partial_x \eta^{\mathbb{T}}_t(x) = 2\sum_{n\geq 1} ne^{-tn^2}(-sin(nx)) \leq 2\sum_{n\geq 1} ne^{-tn^2}
    \end{equation*}

    Noting that, 
    \begin{equation*}
        \begin{cases}
            e^{-tn^2}n \leq e^{-\frac{1}{2}}\frac{1}{\sqrt{2t}}\ \ \ \forall n \leq \frac{1}{\sqrt{2t}},\\
            e^{-tn^2}n \leq e^{-ty^2}y \ \ \ \  \forall n > \frac{1}{\sqrt{2t}}, \ \forall y \in (n-1,n].
        \end{cases}
    \end{equation*}
    So that,
    \begin{equation*}
        |\partial_x\eta^\mathbb{T}_t(x)| \leq e^{-\frac{1}{2}}\frac{1}{2t} + \int_0^\infty e^{-ty^2}ydy = \frac{e^{-1/2}}{2t} + \sqrt{\frac{\pi}{4t}}.
    \end{equation*}
    We conclude, noting $\mathbb{D}$ either $\dR$ or $\mathbb{T}$, that for any $t > 0$,
    \begin{align*}
        \|g_t\|_{L^p_x(L^1_\theta)} &\leq \|\eta^\mathbb{D}\|^{\frac{2}{p}}_{L^1}\|\eta^\mathbb{D}\|^{\frac{2(p-1)}{p}}_{L^\infty}\|\eta^\mathbb{T}\|_{L^1} \leq C_p \left(1+\frac{1}{t^{\frac{p-1}{p}}}\right),\\
        \|\partial_\theta g_t\|_{L^p_x(L^1_\theta)} &\leq \|\eta^\mathbb{D}\|^{\frac{2}{p}}_{L^1}\|\eta^\mathbb{D}\|^{\frac{2(p-1)}{p}}_{L^\infty}\|\partial_\theta\eta^\mathbb{T}\|_{L^1} \leq C_p \left(1+\frac{1}{t^{\frac{p-1}{p} + \frac{1}{2}}}\right),\\
        \max_{i=1,2} \|\partial_{x_i} g_t\|_{L^p_x(L^1_\theta)} &\leq \| \partial_x\eta^\mathbb{D}\|^{\frac{1}{p}}_{L^1}\|\partial_x\eta^\mathbb{D}\|^{\frac{p-1}{p}}_{L^\infty} \|\eta^\mathbb{D}\|^{\frac{1}{p}}_{L^1}\|\eta^\mathbb{D}\|^{\frac{p-1}{p}}_{L^\infty} \|\eta^\mathbb{T}\|_{L^1} \leq C_p\left(1+\frac{1}{t^{\frac{p-1}{p} + \frac{1}{2}}}\right).
    \end{align*}
    
\end{proof}

\subsection*{Gr\"onwall type inequality}

The case $p=\infty$ is a classical, for example, see Henry \cite{henry2006geometric}. We will use the method introduced by Pazy\cite{pazy2012semigroups}, with a finite number of iterations and conclude with the classical Gr\"onwall's Lemma.
Let us, first recall Gr\"onwall's inequality.

\begin{proposition}(Gr\"onwall inequality)
    \label{prop:CGrwnllInq}
    Suppose that $\phi\in L^\infty_+[0,T]$ satisfies the inequality,
    \begin{equation*}
        \phi(t) \leq c_0(t) + \int_0^t c_1(s)\phi(s)ds \text{ for \textit{a.e} } t\in [0,T],
    \end{equation*}
    where $c_1\in L^1_+[0,T]$, and $c_0\in L^\infty_+[0,T]$ is non-deacreasing. Then,
    \begin{equation*}
        \phi(t) \leq c_0(t)\exp\left(\int_0^t c_1(s)ds\right) \text{ for \textit{a.e} } t\in [0,T].
    \end{equation*}
\end{proposition}

In the following, we will need the following identity,
for any $\beta > -1$ and $x\geq 0$,
\begin{equation}
    \label{eq:GrnwIntIdenty}
    \int^x_0 (x-t)^\beta t^\beta dt= \frac{\Gamma(1+\beta)^2}{\Gamma(2+2\beta)}x^{1+2\beta},
\end{equation}
where $\Gamma$ is the gamma function.
\begin{proof}[Proof of Proposition \ref{prop:SGrnwllIneq}]
    Let $4<p\leq \infty$, in the following, if $p=\infty$, we use the convention $\frac{1}{p}=0$. First, we iterate the integral inequality,
    {\small
    \begin{align*}
        \phi(t) & \leq c_0(t) + \int^t_0 \left(1+\frac{1}{(t-s)^{\frac{1}{p} + \frac{1}{2}}}\right) c_1(s) \phi(s)ds,\\
        & \leq c_0(t) + \int^t_0 \left(1+\frac{1}{(t-s)^{\frac{1}{p} + \frac{1}{2}}}\right) c_1(s) \left( c_0(s) + \int^s_0 \left(1+\frac{1}{(s-u)^{\frac{1}{p} + \frac{1}{2}}}\right) c_1(u) \phi(u)du \right)ds\\
        & \leq c_0(t) \underset{=P^1(t)}{\underbrace{\left(1 +  \int^t_0 \left(1+\frac{1}{(t-s)^{\frac{1}{p} + \frac{1}{2}}}\right) c_1(s) ds \right)}}\\
        & \hspace{3em}+ \int^t_0 \int^s_0 \left(1+\frac{1}{(t-s)^{\frac{1}{p} + \frac{1}{2}}}\right)   \left(1+\frac{1}{(s-u)^{\frac{1}{p} + \frac{1}{2}}}\right) c_1(s) c_1(u) \phi(u)duds,\\
    \end{align*}}
    where we used that $c_0$ is increasing.
    Using H\"older inequality we obtain for $4<p<\infty$,
    \begin{equation*}
        P^1(t) = 1 + \|c_1\|_{L^p(0,t)}\left(t^{\frac{p-1}{p}} + C_p t^{\frac{p-4}{2p}}\right),
    \end{equation*}
    and if $p = \infty$,
    \begin{equation*}
        P^1(t) = 1 + \|c_1\|_{L^\infty(0,t)}\left(t + C_p t^{\frac{1}{2}}\right).
    \end{equation*}
    Noting that, for any $\beta \in \dR$,

    {\small
    \begin{equation}
        \label{eq:GrnwllPowerIneq}
        \left(1+x^\beta\right) \left(1+y^\beta\right) \leq (1+ 2t^{|\beta|}) \left(1+x^\beta y^\beta\right) \ \ \forall t>x,y>0.
    \end{equation}}
    
   Applyinng the above with $\beta=-\frac{1}{p}-\frac{1}{2}$, and Fubini-Tonelli Theorem, we obtain,
    {\small
    \begin{align}
        \label{eq:SGrnwllFiteration}
        \phi(t) & \leq c_0(t) P^1(t) +(1+ 2t^{\frac{1}{p} + \frac{1}{2}})\int^t_0 \int^t_u \left(1+\frac{1}{[(t-s)(s-u)]^{\frac{1}{p} + \frac{1}{2}}}\right) c_1(s)ds c_1(u) \phi(u)du,\nonumber\\
        & \leq c_0(t) P^1(t) + (1+ 2t^{\frac{1}{p} + \frac{1}{2}})\int^t_0 \int^{t-u}_0 \left(1+\frac{1}{[(t-u-s)s]^{\frac{1}{p} + \frac{1}{2}}}\right) c_1(s)ds c_1(u) \phi(u)du.
    \end{align}}
    Notice that in the case $p = \infty$, we can already conclude the desired inequality, since using identity \eqref{eq:GrnwIntIdenty},
    
    {\small
    \begin{align}
        \label{eq:GrnwLinfty}
        \phi(t) & \leq c_0(t) P^1(t) + (1 + t^{\frac{1}{2}})\int^t_0 \int^{t-u}_0 \left(1+\frac{1}{[(t-u-s)s]^{\frac{1}{2}}}\right) c_1(s)ds c_1(u) \phi(u)du, \nonumber\\
        & \leq c_0(t) P^1(t) + \underset{\defeq Q^1(t)}{\underbrace{(1 + t^{\frac{1}{2}})\|c_1\|_{L^\infty}(t+\pi)}}\int^t_0 c_1(u) \phi(u)du.
    \end{align}}

    And we obtain, for $p=\infty$, using the classical Gr\"onwall inequality,

    \begin{equation*}
        \phi(t) \leq c(t)M_{\infty}(\|c_1\|_{L^\infty(0,t)}, t) \ \ \ \forall t \in [0,T],
    \end{equation*}

    where $M_\infty$ is a positive non-decreasing continuous function, defined as,
    \begin{equation*}
        M_\infty(\|c_1\|_{L^\infty(0,t)},t) = P^1(t)\exp\left(Q^1(t)t\|c_1\|_{L^\infty(0,t)}\right),
    \end{equation*}
    implying an implicit dependence of $P^1$ and $Q^1$ in $\|c_1\|_{L^\infty(0,t)}$. 
    
    We now continue the iteration procedure for a general $4<p<\infty$. Starting from \eqref{eq:SGrnwllFiteration}, using H\"older inequality, inequality \eqref{eq:GrnwllPowerIneq} and identity \eqref{eq:GrnwIntIdenty},
    {\small
    \begin{align}
        \phi(t) & \leq c_0(t) P^1(t) + (1+ 2t^{\frac{1}{p} + \frac{1}{2}})\|c_1\|_{L^p(0,t)}C_p\int^t_0 (t^\frac{p-1}{p} + (t-u)^{-\frac{3}{p}}) c_1(u) \phi(u)du,\nonumber\\
        & \leq c_0(t) P^1(t) + \underset{\defeq Q^1(t)}{\underbrace{C_p(t^{\frac{p-1}{p}}+1)(1+ 2t^{\frac{1}{p} + \frac{1}{2}})\|c_1\|_{L^p(0,t)}}}\int^t_0 (1 + (t-u)^{-\frac{3}{p}}) c_1(u) \phi(u)du.
    \end{align}}
    Note that we lost some singularity in the kernel, since $-\frac{3}{p}>-\frac{1}{p}-\frac{1}{2}$ for $p>4$.
    We thus introduce the following iteration procedure. At iteration $n\geq 1$, for $r_n > -1$, we have the following integral inequality,

    \begin{equation}
        \label{eq:GrnwlIterate}
        \phi(t) \leq c_0(t) P^n(t) + Q^n(t) \int^t_0 (1 + (t-u)^{r_n}) c_1(u) \phi(u)du.
    \end{equation}

    Iterating \eqref{eq:GrnwlIterate} using H\"older inequality, inequality \eqref{eq:GrnwllPowerIneq} with $\beta=r_n$, Fubini-Tonelli, identity \eqref{eq:GrnwIntIdenty}, and the fact that $Q^n$ and $P^n$ are positive nondecreasing, we obtain, 
    {\small
    \begin{align*}
        \phi(t) & \leq c_0(t) P^n(t) \\
        &\hspace{3em} + Q^n(t) \int^t_0 (1 + (t-s)^{r_n}) c_1(s) \left(c_0(s) P^n(s) + Q^n(s) \int^s_0 (1 + (s-u)^{r_n}) c_1(u) \phi(u)du\right)ds,\\
         &\leq c_0(t) \underset{\defeq P^{n+1}(t)}{\underbrace{P^n(t)\left(1+  Q^n(t) \|c_1\|_{L^p}(t^{\frac{p-1}{p}}+C_{r_n}t^{\frac{p-1}{p}+r_n})\right)}}\\
         &\hspace{3em}+ (Q^n(t))^2 \int^t_0 \int^s_0 (1 + (t-s)^{r_n})  (1 + (s-u)^{r_n})  c_1(s)  c_1(u) \phi(u)duds,\\
         &\leq c_0(t)P^{n+1}(t)+ (Q^n(t))^2 (2t^{|r_n|+1})\int^t_0 \int^s_0 (1 + [(t-s)(s-u)]^{r_n})  c_1(s)  c_1(u) \phi(u)duds,\\
         &\leq c_0(t)P^{n+1}(t)+ \underset{\defeq Q^{n+1}(t)}{\underbrace{C (Q^n(t))^2 (t^{\frac{p-1}{p}}+1)(2t^{|r_n|+1})\|c_1\|_{L^p(0,t)}}}\int^t_0 \left(1 + (t-u)^{\frac{p-1}{p}+2r_n}\right)c_1(u) \phi(u)du,
    \end{align*}}

    for some positive constant $C$ depending on $r_n$ and $p$.
    This establishes recursive relations between $Q^{n+1}, P^{n+1}$ and $Q^n, P^n$, from which it is clear that the positivity and monotonicity of the $(Q^n)$ and $(P^n)$ are preserved. The important relationship is the following,
    \begin{equation}
        \label{eq:GrnwlAritGeometric}
        r_{n+1} = 2r_n + \frac{p-1}{p}.
    \end{equation}
    We thus obtained a similar integral inequality as in \eqref{eq:GrnwlIterate}, with $Q^{n+1}, P^{n+1}$ and $r_{n+1}$. Those recursive relations ensure that $r_n > -1$, and that $Q^{n},P^{n}$ stay positive non-decreasing, so that all the computations above hold.
    From \eqref{eq:GrnwlAritGeometric}, we obtain,
    \begin{equation*}
         r_{n+1} = 2^n\frac{p-4}{p}-\frac{p-1}{p}.
    \end{equation*}
    This implies that for any $4<p<\infty$, there exists a finite integer $n^*\geq 2$, such that $r_{n^*} \geq 0$. Recalling \eqref{eq:GrnwlIterate} for this $n^*$, using the classical Gr\"onwall inequality, we obtain 
    for any $t \in [0,T]$,
    \begin{align*}
        \phi(t) & \leq c_0(t) P^{n^*}(t) + Q^{n^*}(t) \int^t_0 (1 + (t-u)^{r_{n^*}}) c_1(u) \phi(u)du,\\
        & \leq c_0(t) P^{n^*}(t) + Q^{n^*}(t) (1 + t^{r_{n^*}}) \int^t_0 c_1(u) \phi(u)du,\\
        & \leq c_0(t) P^{n^*}(t)\exp\left(Q^{n^*}(t) (1 + t^{r_{n^*}}) t^\frac{p-1}{p} \|c_1\|_{L^p(0,t)}\right).
    \end{align*}

    We conclude the proof, by defining $M_p$ as announced in the statement as follows,
    \begin{equation*}
        M_p(\|c_1\|_{L^p(0,t)},t) \defeq P^{n^*}(t)\exp\left(Q^{n^*}(t) (1 + t^{r_{n^*}}) t^\frac{p-1}{p} \|c_1\|_{L^p(0,t)}\right),
    \end{equation*}
     implying an implicit dependence of $Q^{n^*}$ and $P^{n^*}$ in $\|c_1\|_{L^p(0,t)}$.
\end{proof}